\documentclass[a4paper]{amsart}

\usepackage{amssymb}
\usepackage[utf8]{inputenc}
\usepackage[all]{xy}
\usepackage{mathrsfs,dsfont,nicefrac}
\usepackage{microtype}

\usepackage[pdftitle={Equivariant Lefschetz maps for simplicial complexes and smooth manifolds},
  pdfauthor={Heath Emerson and Ralf Meyer},
  pdfsubject={Mathematics; MSC 19K35, 46L80}
]{hyperref}
\usepackage[lite]{amsrefs}
\newcommand*{\MRref}[2]{\href{http://www.ams.org/mathscinet-getitem?mr=#1}{MR \textbf{#1}}}
\newcommand*{\arxiv}[1]{\href{http://www.arxiv.org/abs/#1}{arXiv: #1}}

\theoremstyle{plain}
\newtheorem{theorem}{Theorem}
\newtheorem{lemma}[theorem]{Lemma}
\newtheorem{proposition}[theorem]{Proposition}
\newtheorem{corollary}[theorem]{Corollary}
\theoremstyle{definition}
\newtheorem{definition}[theorem]{Definition}
\newtheorem{notation}[theorem]{Notation}
\theoremstyle{remark}
\newtheorem{remark}[theorem]{Remark}
\newtheorem{example}[theorem]{Example}

\newcommand*{\C}{\mathbb C}
\newcommand*{\Z}{\mathbb Z}
\newcommand*{\R}{\mathbb R}
\newcommand*{\N}{\mathbb N}

\newcommand*{\Sphere}{\mathbb S}
\newcommand*{\UNIT}{\mathds 1} 

\newcommand*{\dual}{\mathscr P}

\newcommand*{\KK}{\textup{KK}}
\newcommand*{\RKK}{\textup{RKK}}
\newcommand*{\cRKK}{\mathscr{R}\textup{KK}}
\newcommand*{\K}{\textup K}

\DeclareMathOperator{\Lef}{Lef}
\DeclareMathOperator{\Eul}{Eul}
\DeclareMathOperator{\Fix}{Fix}
\DeclareMathOperator{\Hom}{Hom}

\DeclareMathOperator{\Stab}{Stab}
\DeclareMathOperator{\sign}{sign}
\DeclareMathOperator{\supp}{supp}
\DeclareMathOperator{\ind}{ind}
\DeclareMathOperator{\tr}{tr}
\DeclareMathOperator{\Sd}{Sd}
\DeclareMathOperator{\Rep}{Rep}
\DeclareMathOperator{\Deg}{Deg}

\newcommand*{\PD}{\textup{PD}}
\newcommand*{\ID}{{\textup{id}}}
\newcommand*{\ev}{\textup{ev}}
\newcommand*{\op}{\textup{op}}
\newcommand*{\mult}{\textup{n}}

\newcommand*{\selfmap}{\psi} 
\newcommand*{\Selfmap}{\Psi} 
\newcommand*{\col}{\gamma}   

\newcommand*{\Hils}{\mathscr{H}}

\newcommand*{\CONT}{\mathscr{C}}

\newcommand*{\Cliff}{\textup{Cliff}}

\newcommand{\deRham}{\textup{D}_{\textup{dR}}}

\newcommand*{\Comp}{\mathbb{K}}

\newcommand*{\nb}{\nobreakdash}
\newcommand*{\Cstar}{\texorpdfstring{$C^*$\nobreakdash-\hspace{0pt}}{C*-}}
\newcommand*{\Star}{\texorpdfstring{$^*$\nobreakdash-\hspace{0pt}}{*-}}

\newcommand*{\abs}[1]{\lvert#1\rvert}
\newcommand*{\norm}[1]{\lVert#1\rVert}
\newcommand*{\conj}[1]{\overline{#1}}

\newcommand*{\cross}{\mathbin{\rtimes}}
\newcommand*{\defeq}{\mathrel{:=}}
\newcommand*{\blank}{\textup{\textvisiblespace}}

\newcommand*{\prto}{\twoheadrightarrow}

\newcommand*{\no}{\underline{\textbf{n}}}
\newcommand*{\poset}{\mathcal{S}(\no)}

\begin{document}

\title[Equivariant Lefschetz maps]{Equivariant Lefschetz
  maps for simplicial complexes and smooth manifolds}

\author{Heath Emerson}
\email{hemerson@math.uvic.ca}

\address{Department of Mathematics and Statistics\\
  University of Victoria\\
  PO BOX 3045 STN CSC\\
  Victoria, B.C.\\
  Canada V8W 3P4}

\author{Ralf Meyer}
\email{rameyer@uni-math.gwdg.de}

\address{Mathematisches Institut\\
  Georg-August Universit\"at G\"ottingen\\
  Bunsenstra{\ss}e 3--5\\
  37073 G\"ottingen\\
  Germany}

\begin{abstract}
  Let~\(X\) be a locally compact space with a continuous proper
  action of a locally compact group~\(G\).  Assuming that~\(X\)
  satisfies a certain kind of duality in equivariant bivariant
  Kasparov theory, we can enrich the classical construction of
  Lefschetz numbers for self-maps to an equivariant
  K\nobreakdash-homology class.  We compute the Lefschetz
  invariants for self-maps of finite-dimensional simplicial
  complexes and smooth manifolds.  The resulting invariants are
  independent of the extra structure used to compute them.
  Since smooth manifolds can be triangulated, we get two
  formulas for the same Lefschetz invariant in this case.  The
  resulting identity is closely related to the equivariant
  Lefschetz Fixed Point Theorem of L\"uck and Rosenberg.
\end{abstract}

\subjclass[2000]{19K35, 46L80}

\thanks{This research was supported by the National Science and
  Engineering Research Council of Canada Discovery Grant
  program and the Marie Curie Action \emph{Noncommutative
    Geometry and Quantum Groups} (Contract
  MKTD-CT-2004-509794).}

\maketitle

\section{Introduction}
\label{sec:intro}

Euler characteristics and Lefschetz numbers of self-maps are
important objects in algebraic topology.  They can be refined
for spaces with a group action.  A purely topological approach
to such \emph{equivariant} Lefschetz invariants has been
developed by Wolfgang L\"uck and Julia Weber
\cites{Lueck-Rosenberg:Lefschetz, Weber:Universal_Lefschetz}.
This article grew out of the authors' previous work on
equivariant Euler characteristics
in~\cite{Emerson-Meyer:Euler}.  The applications
in~\cite{Emerson-Meyer:Euler} dictated studying Euler
characteristics in the setting of bivariant Kasparov theory.
Here we extend the framework of~\cite{Emerson-Meyer:Euler} to
produce Lefschetz invariants as well.

The main ingredient in our definition of the Lefschetz
invariant is a certain duality in bivariant Kasparov theory.
Let~\(X\) be a locally compact space and let~\(G\) be a locally
compact group acting on~\(X\).  A \(G\)\nb-equivariant
\emph{abstract dual} for~\(X\) is a pair \((\dual,\Theta)\)
consisting of a \(G\)\nb-\Cstar{}algebra~\(\dual\) and a
class \(\Theta\in\RKK^G_n(X;\C,\dual)\) for some \(n\in\N\)
such that Kasparov product with~\(\Theta\) induces isomorphisms
\begin{equation}
  \label{eq:abstractduality}
  \KK^G_*(\dual\otimes A, B) \cong \RKK^G_{*+n}(X; A,B)
\end{equation}
for all \(G\)\nb-\Cstar{}algebras \(A\) and~\(B\); the groups
\(\KK^G_*(\dual\otimes A, B)\) and \(\RKK^G_{*+n}(X; A,B)\)
appearing here are defined by Gennadi Kasparov
in~\cite{Kasparov:Novikov}.

Given an abstract dual for~\(X\), we define an
\emph{equivariant Euler characteristic of~\(X\)} in
\(\KK^G_0(\CONT_0(X),\C)\) and an \emph{equivariant Lefschetz
  map}
\[
\Lef\colon
\RKK^G_*(X; \CONT_0(X),\C) \to \KK^G_*(\CONT_0(X),\C).
\]
The construction of these maps is explained in the body of the
paper, see also~\cite{Emerson-Meyer:Euler}.  The equivariant
Euler characteristic is already studied
in~\cite{Emerson-Meyer:Euler}, the Lefschetz map is studied
in~\cite{Emerson-Meyer:Dualities}, even in the more general
setting of groupoid actions, which includes relative Euler
characteristics and Lefschetz maps for bundles of spaces.  Here
we only consider the simpler case of a single space with a
group action.

In order to actually compute the invariants, it is useful to
have a bit more structure, which is formalised in the notion of
a Kasparov dual in \cites{Emerson-Meyer:Euler,
  Emerson-Meyer:Dualities}.

If~\(X\) is compact, then
\[
\RKK^G_*(X;\CONT_0(X),\C) \cong \KK^G_*(\CONT(X),\CONT(X))
\]
because both groups are defined by the same cycles.  Hence the
domain of the equivariant Lefschetz map consists of morphisms
from~\(X\) to~\(X\) in an appropriate category in this case.
In general, there is a canonical map
\[
\KK^G_*(\CONT_0(X),\CONT_0(X)) \to \RKK^G_*(X;\CONT_0(X),\C)
\]
(see \cite{Emerson-Meyer:Dualities}*{\S4.1.4}).  Roughly
speaking, the group on the right hand side also contains
self-maps of \(\CONT_0(X)\) that are not proper.  In
particular, any \(G\)\nb-equivariant continuous map
\(\selfmap\colon X\to X\) yields a class in
\(\RKK^G_0(X;\CONT_0(X),\C)\), to which we can apply the
Lefschetz map to get \(\Lef(\selfmap)\in
\KK^G_0(\CONT_0(X),\C)\).

We are going to compute \(\Lef(\selfmap)\) in two important
cases where a Kasparov dual is available -- for simplicial
complexes and smooth manifolds -- and compare the results.  An
important point is that \(\Lef(\selfmap)\) does not depend on
the dual that we use to compute it.  Since any smooth manifold
can be triangulated, we therefore get two formulas for
\(\Lef(\selfmap)\).  The equality of these two formulas amounts
to a Lefschetz Fixed Point Theorem in \(\KK^G\).  The result is
comparable to the Equivariant Lefschetz Fixed Point Theorem of
Wolfgang L\"uck and Jonathan Rosenberg
in~\cite{Lueck-Rosenberg:Lefschetz}.  However, there are some
differences.  On the one hand, their invariant is finer than
ours.  On the other hand, we are able to weaken the standard
transversality assumption on~\(\selfmap\) to deal with the case
where the fixed point set of~\(\selfmap\) has strictly positive
dimension.  Instead of orientation data at a fixed point, we
get a contribution of the Euler characteristic of the fixed
point set.  In the transverse case, we describe the local
orientation data in a different way than L\"uck and Rosenberg.
The local data for us is an element of the representation ring
of the isotropy subgroup of the fixed point.  It can be
described using Clifford algebras and their representation
theory.  Finally, our methods allow the group~\(G\) to be
non-discrete.

Lefschetz invariants exist for general elements of
\(\KK^G_*(\CONT_0(X),\CONT_0(X))\), not just for maps.  We will
address this kind of generalisation in a future article.  As in
the Euler characteristic computations
in~\cite{Emerson-Meyer:Euler}, in the analysis presented here,
\(\Lef(\selfmap)\) appears as a linear combination of point
evaluation classes, at least in the simplicial case.  In this
sense, the Lefschetz invariant of a \emph{map} is always a
\(0\)\nb-dimensional object.  This changes when we apply the
Lefschetz map to more general elements in
\(\RKK^G_*(X;\CONT_0(X),\C)\): on that domain, the Lefschetz
map is split surjective by
\cite{Emerson-Meyer:Dualities}*{Proposition 4.26}, so that the
Lefschetz map yields arbitrarily complicated \(\K\)\nb-homology
classes.  Of course, this also makes computations more
difficult.  In another forthcoming article, we shall compute
Lefschetz invariants of geometric cycles or correspondences as
defined by Paul Baum and his coauthors
(\cites{Baum-Block:Bicycles, Baum-Douglas:K-homology}) and by
Alain Connes and Georges
Skandalis~\cite{Connes-Skandalis:Longitudinal}.  But these
computations use a different Kasparov dual in order to express
everything in terms of geometric cycles.  The computations here
are closer to those in~\cite{Emerson-Meyer:Euler}.

There are some similarities between the results here and the
Lefschetz Fixed Point Formula
in~\cite{Echterhoff-Emerson-Kim:Fixed}, which is also based on
the method of Poincar\'e duality.  However, each deal with
different situations: in~\cite{Echterhoff-Emerson-Kim:Fixed}
the trace of certain automorphisms of the crossed product
\(\CONT_0(X)\ltimes G\) on \(\K_*(\CONT_0(X)\ltimes G)\) is
computed.  Our approach here only deals with special
automorphisms that act identically on~\(G\); but it yields a
refined invariant in \(\KK^G_0(\CONT_0(X),\C)\) instead of just
a number.

The contents of this article is as follows.  In
Section~\ref{sec:duals_Lefschetz}, we briefly recall the notion
of a Kasparov dual and define the Lefschetz map.
Section~\ref{sec:main_results} contains our three main results.
Theorem~\ref{the:Lef_combinatorial} provides a formula for
\(\Lef(\selfmap)\) in the simplicial case and
Theorem~\ref{the:Lef_smooth} in the smooth case;
Theorem~\ref{the:Lef_Fixed_Point} combines them into a
Lefschetz Fixed Point Theorem.  The proofs appear in
Section~\ref{sec:combinatorial_proof} for the simplicial case
and in Section~\ref{sec:smooth_proof} for the smooth case.

\section{Kasparov duality and the Lefschetz map}
\label{sec:duals_Lefschetz}

The general framework of duals and Lefschetz maps is explained
carefully in~\cite{Emerson-Meyer:Dualities}.  Here we only
recall the definition of a Kasparov dual for a
\(G\)\nb-space~\(X\), where~\(X\) is a locally compact space
and~\(G\) a locally compact group.  As
in~\cite{Emerson-Meyer:Dualities}, we write \(\UNIT=\C\).  If
we worked with real \Cstar{}algebras, we would use~\(\R\)
instead.

A Kasparov dual for~\(X\) consists of an \(X\cross
G\)-\(C^*\)-algebra~\(\dual\) and two classes
\[
\Theta\in\RKK^G_n(X;\UNIT,\dual),
\qquad
D\in\KK^G_{-n}(\dual,\UNIT),
\]
satisfying some conditions (see
\cite{Emerson-Meyer:Dualities}*{Definition 4.1}).  These ensure
that the maps
\begin{alignat}{2}
  \label{eq:def_PD}
  \PD&\colon \KK^G_{i-n}(A\otimes\dual,B) \to
  \RKK^G_i(X;A,B),
  &\qquad
  f&\mapsto \Theta \otimes_\dual f,\\
  \label{eq:def_PDstar}
  \PD^*&\colon \RKK^G_i(X;A,B) \to
  \KK^G_{i-n}(A\otimes\dual,B),
  &\qquad
  g&\mapsto (-1)^{ni} T_\dual(g) \otimes_\dual D,
\end{alignat}
are inverse to each other.  Here the map~\(T_\dual\) combines
three maps
\begin{multline*}
  \RKK^G_*(X;A,B) =
  \cRKK^G_*\bigl(X;\CONT_0(X,A),\CONT_0(X,B)\bigr)
  \\\to \cRKK^G_*\bigl(X\times X;\CONT_0(X,A)\otimes\dual,
  \CONT_0(X,B)\otimes\dual\bigr)
  \\\to \cRKK^G_*(X;A\otimes\dual, B\otimes\dual)
  \to \KK^G_*(A\otimes\dual, B\otimes\dual);
\end{multline*}
the first map is the exterior product, the second one is
induced by the diagonal embedding \(X\to X\times X\), and the
third one is the forgetful map.  In our computations, we will
only use the following special case:

\begin{lemma}
  \label{lem:compute_T_on_map}
  Let \(\selfmap\colon X\to X\) be a \(G\)\nb-equivariant map.
  As in \cite{Emerson-Meyer:Dualities}*{Example 4.19}, let
  \(\selfmap^*(\Delta_X)\) be the class in
  \(\RKK^G_0(X;\CONT_0(X),\UNIT)\) of the \Star{}homomorphism
  induced by the map \(X\to X\times X\), \(x\mapsto
  \bigl(x,\selfmap(x)\bigr)\).  Then
  \(T_\dual(\selfmap^*\Delta_X)\in \KK^G_0(
  \CONT_0(X)\otimes\dual, \dual)\) is represented by the
  \(G\)\nb-equivariant \Star{}homomorphism
  \[
  \mu_\selfmap\colon \CONT_0(X)\otimes\dual \to\dual,
  \qquad
  f\otimes a \mapsto (f\circ \selfmap)\cdot a,
  \]
  where the multiplication uses the \(X\)\nb-structure
  on~\(\dual\).
\end{lemma}

\begin{proof}
  The proof is straightforward.
\end{proof}

The \emph{Lefschetz map} for~\(X\) is the composition
\begin{multline}
  \label{lefmap}
  \Lef\colon \RKK^G_*(X;\CONT_0(X),\UNIT)
  \xrightarrow{\PD^*}
  \KK^G_{*+n}(\CONT_0(X)\otimes\dual, \UNIT)
  \\\xrightarrow{\overline{\Theta}
    \otimes_{\CONT_0(X)\otimes\dual}\blank}
  \KK^G_*(\CONT_0(X), \UNIT).
\end{multline}
Here \(\overline{\Theta}\) is obtained from~\(\Theta\) by
applying the forgetful map
\[
\RKK^G_n(X;\UNIT,\dual) \to
\KK^G_n(\CONT_0(X),\CONT_0(X)\otimes\dual).
\]

Given a \(G\)\nb-equivariant map \(f\colon X\to X\), we
define
\[
\Lef(\selfmap) \defeq \Lef(\selfmap^*\Delta_X),
\]
where we use \(\selfmap^*\Delta_X\in
\RKK^G_0(X;\CONT_0(X),\UNIT)\) as defined in
Lemma~\ref{lem:compute_T_on_map}.  Plugging in
\eqref{eq:def_PDstar} and Lemma~\ref{lem:compute_T_on_map},
this becomes
\begin{equation}
  \label{eq:compute_Lef}
  \Lef(\selfmap)
  = \overline{\Theta} \otimes_{\CONT_0(X)\otimes\dual}
  T_\dual(\selfmap^*\Delta_X)\otimes_{\dual} D
  = \overline{\Theta} \otimes_{\CONT_0(X)\otimes\dual}
  [\mu_\selfmap]\otimes_\dual D.
\end{equation}
Since the class of \(\selfmap^*\Delta_X\) in
\(\RKK^G_0(X;\CONT_0(X),\UNIT)\) depends only on the
equivariant homotopy class of~\(\selfmap\), so does
\(\Lef(\selfmap)\).  The Lefschetz invariant of the identity
map is the \emph{equivariant Euler characteristic} of~\(X\):
\begin{equation}
  \label{eq:Euler_as_Lef}
  \Eul_X \defeq \Lef(\ID_X) \in \KK^G_0(\CONT_0(X), \UNIT).  
\end{equation}
This agrees with the definition is~\cite{Emerson-Meyer:Euler}.

All choices of a Kasparov dual \((\dual,\Theta,D)\) yield the
same Lefschetz map.  This is shown
in~\cite{Emerson-Meyer:Euler} for the equivariant Euler
characteristic; as observed in~\cite{Emerson-Meyer:Dualities},
the proof generalises to the Lefschetz map.  We will use this
fact when~\(X\) is, at the same time, a smooth Riemannian
manifold with~\(G\) acting isometrically and a simplicial
complex with~\(G\) acting simplicially.  Both the smooth and
the combinatorial structure provide Kasparov duals for~\(X\).
These two duals yield different formulas for
\(\Lef(\selfmap)\).  But they must produce the same class in
\(\KK^G_0(\CONT_0(X),\UNIT)\).  The equality between both
results is our equivariant Lefschetz formula.

\section{Statement of the main results}
\label{sec:main_results}

We first consider the simplest possible case.  We let~\(G\) be
trivial and let~\(X\) be a connected compact manifold equipped
with a triangulation.  Let \(C_\bullet(X)\) be the resulting
simplicial chain complex.  Let \(\selfmap\colon X\to X\) be a
self-map.

We can find a smooth map \(\tilde\selfmap\colon X\to X\) that is
homotopic to~\(\selfmap\) and whose graph is transverse to the
diagonal.  When we use the Kasparov dual from the smooth
structure on~\(X\), we get
\[
\Lef(\selfmap) = \Lef(\tilde\selfmap) =
\sum_{p\in \Fix(\tilde\selfmap)} \sign \det
\bigl(\ID_{T_pX}-D_p\tilde\selfmap\bigr) \cdot [\ev],
\]
where \(D_p\tilde\selfmap\) denotes the derivative
of~\(\tilde\selfmap\) at~\(p\) and \([\ev]\in \K_0(X)\) denotes
the class of point evaluations.  Since~\(X\) is connected and
\(\KK\)-theory is homotopy invariant, all point evaluations
have the same \(\K\)\nb-homology class.  The numerical factor
in front of \([\ev]\) is the usual local formula for the
Lefschetz number of~\(\tilde\selfmap\) in terms of fixed
points.

The triangulation provides another Kasparov dual for~\(X\).
This dual yields
\[
\Lef(\selfmap) = \sum_{d=0}^{\dim X} (-1)^d \tr
\bigl(\Selfmap_d \colon C_d(X) \to C_d(X)\bigr)\cdot [\ev],
\]
where \(\Selfmap_\bullet\colon C_\bullet(X) \to C_\bullet(X)\)
is the chain map induced by~\(\selfmap\) (or a cellular
approximation of~\(\selfmap\)).  A familiar trick replaces the
spaces of cycles \(C_d(X)\) by the homology spaces
\(\textup{H}_d(X)\) in the last formula, expressing the
Lefschetz number as a global homological invariant.

The equivariant generalisation of our Lefschetz formulas
requires some preliminary notation.  First we explain the
formula for the combinatorial case, then for the smooth case.

\subsection{The combinatorial Lefschetz map}
\label{sec:combinatorial_case}

Let~\(X\) be a finite-dimensional simplicial complex and
let~\(G\) be a locally compact group acting smoothly and
simplicially on~\(X\) (that is, stabilisers of points are
open).  The finite-dimensionality assumption is needed for the
construction of a Kasparov dual in~\cite{Emerson-Meyer:Euler}
to work.  It can probably be dropped if we use a more
sophisticated dual.  Assume that~\(X\) admits a colouring (that
is, \(X\) is typed) and that~\(G\) preserves the colouring.
This ensures that if \(g\in G\) maps a simplex to itself, then
it fixes that simplex pointwise.

Let \(SX\) be the set of (non-degenerate) simplices of~\(X\)
and let \(S_dX\subseteq SX\) be the subset of
\(d\)\nb-dimensional simplices.  The group~\(G\) acts smoothly
on the discrete set~\(SX\) preserving the decomposition
\(SX=\bigsqcup S_d X\).  Decompose~\(SX\) into \(G\)\nb-orbits.
For each orbit \(\dot\sigma \subseteq SX\), choose a
representative \(\sigma\in SX\) and let \(\xi_\sigma\in X\) be
its barycentre and \(\Stab(\sigma)\subseteq G\) its stabiliser.
Restriction to the orbit of~\(\xi_\sigma\) defines a
\(G\)\nb-equivariant \Star{}homomorphism
\begin{equation}
  \label{eq:dot_xi}
  \xi_{\dot\sigma}\colon
  \CONT_0(X)\to\CONT_0\bigl(G/\Stab(\sigma)\bigr) \to
  \Comp\bigl(\ell^2(G/\Stab\sigma)\bigr),
\end{equation}
where the second map is the representation by pointwise
multiplication operators.  We let \([\xi_{\dot\sigma}]\) be its
class in \(\KK^G_0(\CONT_0(X),\UNIT)\).

Let \(\selfmap\colon X\to X\) be a \(G\)\nb-equivariant
self-map of~\(X\).  It is too restrictive to
assume~\(\selfmap\) to be simplicial -- simplicial
approximation requires us to refine the triangulations on
domain and target independently, so that we may need two
different triangulations of~\(X\) for a self-map.  What we can
achieve is that~\(\selfmap\) becomes a cellular map with
respect to the canonical CW-complex structure on a simplicial
complex.

More precisely, \(\selfmap\) is \(G\)\nb-equivariantly
homotopic to a \(G\)\nb-equivariant cellular map.  Hence we may
assume without loss of generality that~\(\selfmap\) is itself
cellular, so that it induces a \(G\)\nb-equivariant chain map
\[
\Selfmap\colon C_\bullet(X)\to C_\bullet(X),
\]
where \(C_\bullet (X)\) is the cellular chain complex of~\(X\);
this is nothing but the chain complex of \emph{oriented}
simplices of~\(X\).  A basis for \(C_\bullet (X)\) is given by
the set of (un)oriented simplices, by arbitrarily choosing an
orientation on each simplex.  We may describe the chain
map~\(\Selfmap\) by its matrix coefficients
\(\Selfmap_{\sigma\tau}\in\Z\) with respect to this basis; thus
the subscripts are unoriented simplices.  For example,
if~\(\selfmap\) maps a simplex to itself and reverses its
orientation, then \(\Selfmap_{\sigma,\sigma}=-1\).
Since~\(\selfmap\) is \(G\)\nb-equivariant,
\(\Selfmap_{g(\sigma),g(\sigma)} = \Selfmap_{\sigma\sigma}\).
So the following makes sense.

\begin{notation}
  \label{note:mult_gamma}
  For \(\dot{\sigma} \in G\backslash S_dX\), let
  \(\mult(\Selfmap,\dot\sigma) \defeq
  (-1)^d\Selfmap_{\sigma\sigma} \in \Z\) for any choice of
  representative \(\sigma \in \dot{\sigma}\).
\end{notation}

\begin{theorem}
  \label{the:Lef_combinatorial}
  Let~\(X\) be a finite-dimensional coloured simplicial
  complex and let~\(G\) be a locally compact group that acts
  smoothly and simplicially on~\(X\), preserving the colouring.
  Let \(\selfmap\colon X\to X\) be a \(G\)\nb-equivariant
  self-map.  Define \(\mult(\Selfmap,\dot\sigma)\in\Z\) and
  \([\xi_{\dot\sigma}]\in \KK^G_0(\CONT_0(X),\UNIT)\) for
  \(\dot\sigma\in G\backslash SX\) as above.  Then
  \[
  \Lef(\selfmap) = \sum_{\dot\sigma\in G\backslash SX}
  \mult(\Selfmap,\dot\sigma)[\xi_{\dot\sigma}].
  \]
\end{theorem}

Theorem~\ref{the:Lef_combinatorial} will be proved in
\S\ref{sec:combinatorial_proof}, where we also recall the
Kasparov dual for~\(X\).

\subsection{The \texorpdfstring{$\K$}{K}-orientation of a
  vector bundle automorphism}
\label{sec:orient_vb_auto}

Here we prepare for the Lefschetz formula for self-maps of
manifolds.  We need an invariant for vector bundle
automorphisms that generalises the sign of the determinant of a
linear map \(A\colon V\to V\) on a single vector space.  This
is a refinement of a construction which appears
in~\cite{Echterhoff-Emerson-Kim:Fixed} in connection with the
equivariant index of twisted Schr\"odinger operators.  If~\(G\)
is a group of orthogonal transformations of~\(\R^n\) and if
\(A\in\textup{Gl}_n(\R)\) commutes with~\(G\), then a
refinement of the sign of the determinant of~\(A\) is the
virtual group character
\begin{equation}
  \label{indexcharacter}
  \chi_{G,A}\colon G \to \{\pm1\},
  \qquad
  \chi_{G,A}(g) \defeq \sign \det (A|_{\Fix(g)}).
\end{equation}
The virtual representation of~\(G\) corresponding to
\(\chi_{G,A}\) is the solution to an equivariant index problem
on~\(\R^n\).  Here we take a different approach to the same
invariant, which permits a generalisation of it to equivariant
bundles. In the case where~\(Y\) is a point, the equivalence
between the picture in~\cite{Echterhoff-Emerson-Kim:Fixed} and
the one presented here follows from the fact that both give a
solution to the same equivariant index problem.

A very similar invariant to ours appears in the work of L\"uck
and Rosenberg, but the connection to index theory and to
representation theory seems absent.

We consider the following situation.  Let~\(Y\) be a locally
compact space, let~\(G\) be a locally compact group (or
groupoid) acting continuously on~\(Y\), and let \(\pi\colon
E\to Y\) be a \(G\)\nb-equivariant Euclidean \(\R\)\nb-vector
bundle over~\(E\), that is, \(E\) comes equipped with a
\(G\)\nb-invariant inner product on its fibres.  Let \(A\colon
E\to E\) be a \(G\)\nb-equivariant vector bundle automorphism,
that is, a \(G\)\nb-equivariant continuous map \(E\to E\)
over~\(Y\) that restricts to \(\R\)\nb-vector space
isomorphisms on the fibres of~\(E\).  We are going to define a
\(G\)\nb-equivariant \(\Z/2\)-graded real line bundle
\(\sign(A)\) over~\(Y\).

Since we work with complex \(\K\)\nb-theory most of the time,
we are mainly interested in the complexification
\(\sign(A)\otimes_\R\C\).  Nevertheless, it is worth noting
that the line bundles we get are complexifications of real line
bundles.

If~\(Y\) is a point, then \(G\)\nb-equivariant \(\Z/2\)-graded
real vector bundles over~\(Y\) correspond to real orthogonal
virtual representations of~\(G\).  The initial data is thus a
real orthogonal representation of~\(G\) on a Euclidean
space~\(\R^n\) and an invertible linear map \(A\colon
\R^n\to\R^n\) commuting with~\(G\).  The sign is a virtual
\(1\)\nb-dimensional representation of~\(G\) and hence
corresponds to a pair \((\chi,\varepsilon)\), where
\(\varepsilon\in\{0,1\}\) is the parity of the line bundle and
\(\chi\colon G\to \{-1,+1\}\) is a real-valued character.  The
parity~\(\varepsilon\) turns out to be~\(0\) if~\(A\) preserves
orientation and~\(1\) if~\(A\) reverses orientation (see
Example~\ref{exa:orientation}).  In this sense, our invariant
refines the orientation of~\(A\).  As it happens, if the parity
is even, then \(\chi = \chi_{G,A}\) as
in~\eqref{indexcharacter}, and if the parity is odd, then
\(\chi = -\chi_{G,A}\). In particular, evaluating \(\chi_{G,
  A}\) at the identity of~\(G\) yields the parity.

\begin{notation}
  \label{note:Cliff}
  Let \(\Cliff(E)\) be the bundle of Clifford algebras
  associated to~\(E\); its fibre \(\Cliff(E)_x\) is the (real)
  Clifford algebra of~\(E_x\) with respect to the given inner
  product.
\end{notation}

We can also form \(\Cliff(E)\) if~\(E\) carries an indefinite
bilinear form.  If the index of the bilinear form on~\(E\) is
divisible by~\(8\), then the fibres of \(\Cliff(E)\) are
isomorphic to matrix algebras.  In this case, a
\emph{\(G\)\nb-equivariant spinor bundle} for~\(E\) is a
\(\Z/2\)-graded real vector bundle~\(S_E\) together with a
grading preserving, \(G\)\nb-equivariant \Star{}algebra
isomorphism \(c\colon \Cliff(E)\to \operatorname{End}(S_E)\).
This representation is determined uniquely by its restriction
to \(E\subseteq \Cliff(E)\), which is a \(G\)\nb-equivariant
map \(c\colon E\to \operatorname{End}(S_E)\) such that \(c(x)\)
is odd and symmetric and satisfies \(c(x)^2 = \norm{x}^2\) for
all \(x\in E\).  We only use spinor bundles in this special
case.

Recall that the spinor bundle is unique up to tensoring with a
\(G\)\nb-equivariant real line bundle~\(L\): if \(c_t\colon
E\to S_t\) for \(t=1,2\) are two \(G\)\nb-equivariant spinor
bundles for~\(E\), then we define a \(G\)\nb-equivariant real
line bundle~\(L\) over~\(Y\) by
\[
L\defeq \Hom_{\Cliff(E)}(S_1,S_2),
\]
and the evaluation isomorphism \(S_1\otimes L
\xrightarrow{\cong} S_2\) intertwines the representations
\(c_1\) and~\(c_2\) of \(\Cliff(E)\).

\begin{definition}
  \label{def:sign_A_K-oriented}
  Let \(A\colon E\to E\) be a \(G\)\nb-equivariant vector
  bundle automorphism and let \(A=T\circ
  (A^*A)^{\nicefrac{1}{2}}\) be its polar decomposition with an
  orthogonal vector bundle automorphism \(T\colon E\to E\).

  Let~\(F\) be another \(G\)\nb-equivariant vector bundle
  over~\(Y\) with a non-degenerate bilinear form, such that the
  signature of \(E\oplus F\) is divisible by~\(8\), so that
  \(\Cliff(E\oplus F)\) is a bundle of matrix algebras
  over~\(\R\) and \(E\oplus F\) has a \(G\)\nb-equivariant
  spinor bundle, that is, there exists a \(G\)\nb-equivariant
  linear map \(c\colon E\oplus F\to \operatorname{End}(S)\)
  that induces an isomorphism of graded \Star{}algebras
  \[
  \Cliff(E\oplus F) \cong \operatorname{End}(S).
  \]
  Then
  \[
  c'\colon E\oplus F\to\operatorname{End}(S), \qquad
  (\xi,\eta)\mapsto c(T(\xi),\eta)
  \]
  yields another \(G\)\nb-equivariant spinor bundle for
  \(E\oplus F\).  We let
  \[
  \sign(A) \defeq
  \Hom_{\Cliff(E\oplus F)}\bigl((S,c'),(S,c)\bigr).
  \]
  This is a \(G\)\nb-equivariant \(\Z/2\)-graded real line
  bundle over~\(Y\).
\end{definition}

\begin{example}
 \label{exa:sign_ID}
 For example, consider \(A = \ID\).  Choose \(F\), \(S\) and
 \(c\) as in the definition.  The sign of \(\ID\) subject to
 these choices is \(\Hom_{\Cliff(E\oplus
   F)}\bigl((S,c),(S,c)\bigr)\).  Since the spinor
 representation is fibrewise irreducible, the only fibrewise
 endomorphisms of~\(S\) that commute with the Clifford action
 are multiples of the identity map.  The identity map is
 grading-preserving and commutes with~\(G\).  Hence
 \(\sign(\ID)\) is the trivial, evenly graded line bundle
 over~\(Y\) equipped with trivial \(G\)\nb-action.
\end{example}

Formally, we can think of the sign construction as follows. The
set of equivariant \(\K\)\nb-orientations on~\(X\) is in a
natural way a module over the Abelian group of real equivariant
line bundles.  The procedure given above of twisting with an
equivariant map \(f\colon X \to X\) is invariant under twisting
with a real equivariant line bundle, that is, it commutes with
the module structure.  Hence it must itself be given by module
product with some equivariant real line bundle.  The latter is
the sign of~\(f\).

\begin{lemma}
  \label{lem:sign_A_well-defined}
  The \(G\)\nb-equivariant \(\Z/2\)-graded real line bundle
  \(\sign(A)\) is well-defined, that is, it depends neither on
  the bundle~\(F\) nor on the spinor bundle~\(S\) for \(E\oplus
  F\).  Furthermore, \(\sign(A)\) is a homotopy invariant
  of~\(A\) and has the following additivity properties:
  \begin{itemize}
  \item \(\sign(A_1\circ A_2) \cong
    \sign(A_1)\otimes\sign(A_2)\) for two equivariant
    automorphisms \(A_1,A_2\colon E\rightrightarrows E\) of the
    same bundle;

  \item \(\sign(A_1\oplus A_2) \cong
    \sign(A_1)\otimes\sign(A_2)\) for two equivariant vector
    bundle automorphisms \(A_1\colon E_1\to E_1\) and
    \(A_2\colon E_2\to E_2\).

  \end{itemize}
\end{lemma}

\begin{proof}
  Since the spinor bundle for \(E\oplus F\) is unique up to
  tensoring with some line bundle~\(L\), changing the spinor
  bundle replaces \(\sign(A)\) by \(L \otimes \sign(A) \otimes
  L^*\) for some line bundle, which is canonically isomorphic
  to \(\sign(A)\).  Homotopy invariance of \(\sign(A)\) follows
  because homotopic line bundles are isomorphic, and this
  implies that it suffices to treat orthogonal
  transformations~\(A\).

  Let \(E_1\) and~\(E_2\) be two \(G\)\nb-equivariant vector
  bundles with \(G\)\nb-equivariant spinor bundles
  \((S_1,c_1)\) and \((S_2,c_2)\).  Let \(A_1\colon E_1\to
  E_1\) and \(A_2\colon E_2\to E_2\) be two vector bundle
  automorphisms.  Equip \(E\defeq E_1\oplus E_2\) with the
  automorphism \(A\defeq A_1\oplus A_2\) and the
  \(G\)\nb-equivariant spinor bundle \(S\defeq S_1\otimes
  S_2\) with the representation
  \[
  c\defeq (c_1\otimes 1)\oplus (1\otimes c_2)\colon
  E_1\oplus E_2\to \operatorname{End}(S_1\otimes S_2).
  \]
  The canonical map
  \begin{multline*}
    \Hom_{\Cliff(E_1)}\bigl((S_1,c_1), (S_1,c_1\circ A_1)\bigr)
    \otimes
    \Hom_{\Cliff(E_2)}\bigl((S_2,c_2), (S_2,c_2\circ A_2)\bigr)
    \\\to
    \Hom_{\Cliff(E)}\bigl((S,c), (S,c\circ A)\bigr)
  \end{multline*}
  is an isomorphism.

  For \(A_2=\ID\), this shows that \(\sign(A_1)\) remains
  unchanged if we stabilise by another vector bundle with a
  spin structure.  Hence \(\sign(A)\) does not depend on~\(F\).
  We also get the second additivity property; the first one is
  trivial.
\end{proof}

\begin{example}
  \label{exa:orientation}
  By~\(\R\) we will understand the \(1\)\nb-dimensional real
  vector space with positive parity, and by~\(\R^\op\) the same
  vector space with odd parity.

  By the definitions, if~\(Y\) is a point and~\(G\) is trivial,
  then \(\sign(A)\) is either~\(\R\) or~\(\R^\op\).  Since
  \(\sign(A)\) is homotopy invariant, all
  orientation-preserving maps~\(A\) have
  \(\sign(A)=\sign(\ID)=\R\).  A routine computation (see the
  next example) shows that the sign of the orientation reversal
  automorphism \(x\mapsto -x\) on~\(\R\) is~\(\R^\op\).  Since
  any orientation-reversing map~\(A\) is homotopic to the
  direct sum of \(x\mapsto -x\) and \(\ID_{\R^{n-1}}\), it
  follows from Lemma~\ref{lem:sign_A_well-defined} that
  \(\sign(A) = \R\) for orientation-preserving~\(A\) and
  \(\sign (A) = \R^\op\) for orientation-reversing~\(A\), as
  claimed above.
\end{example}

\begin{example}
  \label{exa:orientation_representation}
  Consider \(G=\Z/2\).  We use similar notation as above, but
  decorate~\(\R\) (or~\(\R^\op\)) by a subscript which is a
  character, as appropriate.

  Let \(\tau\colon G\to\{1\}\) be the trivial character and let
  \(\chi\colon G\to \{+1,-1\}\) be the non-trivial character.
  Consider \(A\colon \R_\chi\to\R_\chi\), \(t\mapsto -t\).
  Then \(\sign(A)\cong\R_\chi^\op\) carries a non-trivial
  representation.

  To see this, let~\(F\) be~\(\R_\chi\) with
  negative definite metric.  Thus the Clifford algebra of
  \(\R_\chi\oplus\R_\chi\) is \(\Cliff_{1,1} \cong
  \mathbb{M}_{2\times2}(\R)\).  Explicitly, the map
  \[
  c(x,y)
  =\begin{pmatrix}0&x-y\\x+y&0\end{pmatrix}
  \]
  induces the isomorphism.  We equip~\(\R^2\) with the
  representation \(\tau\oplus\chi\), so that~\(c\) is
  equivariant.

  Twisting by~\(A\) yields another representation
  \[
  c'(x,y)\defeq c(-x,y) = Sc(x,y)S^{-1}
  \qquad\text{with} \quad S=S^{-1}=
  \begin{pmatrix}0&1\\-1&0\end{pmatrix}.
  \]
  Since~\(S\) reverses the grading and exchanges the
  representations \(\tau\) and~\(\chi\), it induces an
  isomorphism \((\R_\tau\oplus\R_\chi^\op)\otimes \R_\chi^\op
  \xrightarrow{\cong} \R_\tau\oplus\R_\chi^\op\).  Hence
  \(\sign(A) = \R_\chi^\op\).

  This result can be computed using the
  picture~\eqref{indexcharacter} instead.  The fixed point set
  of the identity element of~\(\Z/2\) is all of~\(\R\).  The
  fixed point set of the non-trivial generator of \(\Z/2\) is
  \(\{0\}\subset \R\).  Restricting~\(A\) to these subspaces
  gives \(\sign \det\) equal to \(-1\) and~\(1\), respectively.
  This describes the virtual character~\(-\chi\).
\end{example}

Now we comment on the relationship between our equivariant
orientation and the corresponding notion used by L\"uck and
Rosenberg in~\cite{Lueck-Rosenberg:Lefschetz}.  Our work
intersects with theirs when~\(Y\) is discrete, that is, we are
dealing with a self-map with isolated, regular fixed points.
Since both methods use induction in the same way, we will just
consider the case when~\(Y\) is a point and~\(G\) is a finite
group.

Thus~\(G\) acts by orthogonal transformations on a Euclidean
space~\(E\), and~\(A\) is a \(G\)\nb-equivariant, invertible
linear map on~\(E\).  We have shown above how to associate to
this data a virtual character \(\sign(A)\) of~\(G\), which is
described in~\eqref{indexcharacter}.  To the same data, L\"uck
and Rosenberg associate an invariant \(\Deg_0^G(A)\) in the
Abelian group \(U^G(\star)\) of \(\Z\)\nb-valued functions on
the set \(\Pi(G)\) of conjugacy classes of subgroups of~\(G\).
Computations in~\cite{Lueck-Rosenberg:Lefschetz} describe
\(\Deg_0^G(A)\) by the formula
\[
\Deg_0^G(A)(L) = \sign \det (A|_{\Fix(L)}).
\]
Comparison with~\eqref{indexcharacter} shows that this
restricts to our virtual character on cyclic subgroups.  Thus
their local invariant contains more information than ours.

\subsection{The smooth Lefschetz map}
\label{sec:smooth_case}

Now let~\(X\) be a smooth Riemannian manifold and assume
that~\(G\) acts on~\(X\) isometrically and continuously.

Let \(\selfmap\colon X\to X\) be a \(G\)\nb-equivariant
self-map of~\(X\).  In order to write down an explicit local
formula for \(\Lef(\selfmap)\), we impose the following
restrictions on~\(\selfmap\):
\begin{itemize}
\item \(\selfmap\) is smooth;

\item the fixed point subset \(\Fix(\selfmap)\) of~\(\selfmap\)
  is a submanifold of~\(X\);

\item if \((p,\xi)\in TX\) is fixed by the derivative
  \(D\selfmap\), then~\(\xi\) is tangent to \(\Fix(\selfmap)\).
\end{itemize}
The last two conditions are automatic if~\(\selfmap\) is
isometric with respect to some Riemannian metric (not
necessarily the given one) and hence if~\(\selfmap\) has finite
order.

In the simplest case, \(\selfmap\) and \(\ID_X\) are
transverse, that is, \(\ID-D\selfmap\) is invertible at each
fixed point of~\(\selfmap\); this implies that~\(\selfmap\) has
isolated fixed points.  While this situation can always be
achieved in the non-equivariant case, we cannot expect
transversality for non-discrete~\(G\) because the fixed point
set must be \(G\)\nb-invariant.

To describe the Lefschetz invariant, we abbreviate \(Y\defeq
\Fix(\selfmap)\).  This is a closed submanifold of~\(X\) by
assumption.  Let~\(\nu\) be the normal bundle of~\(Y\)
in~\(X\).  Since the derivative \(D\selfmap\) fixes the tangent
space of~\(Y\), it induces a linear map \(D_\nu\selfmap\colon
\nu\to\nu\).  By assumption, the map
\(\ID_\nu-D_\nu\selfmap\colon \nu\to\nu\) is invertible.

\begin{theorem}
  \label{the:Lef_smooth}
  Let~\(X\) be a complete smooth Riemannian manifold, let~\(G\)
  be a locally compact group that acts on~\(X\) continuously
  and by isometries, and let \(\selfmap\colon X\to X\) be a
  smooth self-map.  Assume that the fixed point subset
  \(Y\defeq \Fix(\selfmap)\) of~\(\selfmap\) is a submanifold
  of~\(X\) and that all tangent vectors of~\(X\) fixed by
  \(D\selfmap\) are tangent to~\(Y\).

  Let~\(\nu\) be the normal bundle of~\(Y\) in~\(X\) and let
  \(D_\nu\selfmap\colon \nu\to\nu\) be induced by the
  derivative of~\(\selfmap\) as above.  Let \(r_Y\colon
  \CONT_0(X)\to\CONT_0(Y)\) be the restriction map and let
  \(\Eul_Y\in\KK^G_0(\CONT_0(Y),\UNIT)\) be the equivariant
  \(\K\)\nb-homology class of the de Rham operator on~\(Y\).
  Then
  \[
  \Lef(\selfmap) =
  r_Y \otimes_{\CONT_0(Y)} \sign(\ID_\nu- D_\nu\selfmap)
  \otimes_{\CONT_0(Y)} \Eul_Y.
  \]
\end{theorem}

\begin{remark}
  \label{rem:Lef_ID}
  It is shown in~\cite{Emerson-Meyer:Euler} that the
  equivariant Euler characteristic \(\Eul_X \defeq \Lef
  (\ID_X)\) is the class of the de Rham operator on~\(X\).
  This justifies the notation \(\Eul_Y\) in
  Theorem~\ref{the:Lef_smooth}.  If~\(\selfmap\) is the
  identity map, then Theorem~\ref{the:Lef_smooth} reduces to
  this description of \(\Eul_X\).
\end{remark}

Now we make Theorem~\ref{the:Lef_smooth} more concrete in the
special case where~\(\selfmap\) and~\(\ID_X\) are transverse.
Then the fixed point subset~\(Y\) is discrete. A discrete set
is a manifold, and Theorem~\ref{the:Lef_smooth} describes its
Euler characteristic: the de Rham operator on~\(Y\) is the zero
operator on \(L^2(\Lambda^*_\C(T^*X)) = \ell^2(Y)\), so that we
have the class of the representation
\(\CONT_0(Y)\to\Comp(\ell^2Y)\) by pointwise multiplication
operators.

The normal bundle~\(\nu\) to~\(Y\) in~\(X\) is the restriction
of~\(TX\) to~\(Y\).  For \(p\in Y\), let~\(n_p\) be~\(+1\) if
\(\ID_{T_pX}-D_p\selfmap\) preserves orientation, and~\(-1\)
otherwise.  The graded equivariant line bundle
\(\sign(\ID_\nu-D_\nu\selfmap)\) in
Theorem~\ref{the:Lef_smooth} is determined by pairs
\((n_p,\chi_p)\) for \(p\in Y\), where~\(n_p\) is the parity of
the representation at~\(p\) and~\(\chi_p\) is a certain
real-valued character \(\chi_p\colon \Stab(p)\to\{-1,+1\}\)
that depends on \(\ID_{T_pX}-D_p\selfmap\) and the
representation of the stabiliser \(\Stab(p)\subseteq G\) on
\(T_pX\) via the formula~\eqref{indexcharacter}.  Equivariance
implies that~\(n_p\) is constant along \(G\)\nb-orbits,
whereas~\(\chi_p\) behaves like \(\chi_{g\cdot p} = \chi_p\circ
\operatorname{Ad}(g^{-1})\).  Let \(\ell^2_\chi(Gp)\) be the
representation of \(G\ltimes\CONT_0\bigl(G/\Stab(p)\bigr)\)
obtained by inducing the representation~\(\chi_p\) from
\(\Stab(p)\), and let \(\CONT_0(X)\) act on \(\ell^2_\chi(Gp)\)
by restriction to \(G/\Stab(p)\).  This defines a
\(G\)\nb-equivariant \Star{}homomorphism
\[
\xi_{Gp,\chi}\colon \CONT_0(X)\to \Comp(\ell^2_\chi G).
\]
Theorem~\ref{the:Lef_smooth} implies

\begin{corollary}
  If the graph of~\(\phi\) is transverse to the diagonal in
  \(X\times X\), then
  \[
  \Lef(\selfmap) = \sum_{Gp\in G\backslash\Fix(\selfmap)}
  n_p [\xi_{Gp,\chi}]
  \]
  where \([\xi_{Gp,\chi}] \in \KK^G_0(\CONT_0(X),\UNIT)\) and
  the multiplicities~\(n_p\) are explained above.

  Furthermore, the character \(\chi\colon \Stab_G(p) \to
  \{-1,+1\}\) at a fixed point~\(p\) has the explicit formula
  \[
  \chi(g) = \sign \det \bigl(\ID - D_p\phi\bigr)
  \cdot \sign \det \bigl( \ID - D_p\phi|_{\Fix(g)}\bigr).
  \]
\end{corollary}

If, in addition, \(G\) is trivial and~\(X\) is connected, then
\(\xi_{Gp,\chi}=\ev_p\) for all \(p\in Y\); moreover, all point
evaluations have the same \(\K\)\nb-homology class because
they are homotopic.  Hence we get the classical Lefschetz data
multiplied by the \(\K\)\nb-homology class of a point
\[
\Lef(\selfmap) =
\biggl(\sum_{p\in\Fix(\selfmap)}
\sign(\ID_{T_pX} - D\selfmap_p)\biggr) \cdot [\ev]
\]
as asserted above.  This sum is finite if~\(X\) is compact.

\begin{lemma}
  \label{pointevals}
  Let \(H\subseteq G\) be compact and open, let \(p,q\in X^H\)
  belong to the same path component of the fixed point
  subspace~\(X^H\), and let \(\chi\in\Rep(H)\).  Then
  \[
  \bigl[\xi_{Gp,\ind_H^{\Stab(Gp)}(\chi)}\bigr] =
  \bigl[\xi_{Gq,\ind_H^{\Stab(Gq)}(\chi)}\bigr]
  \qquad \text{in \(\KK_0^G(\CONT_0(X),\UNIT)\).}
  \]
\end{lemma}

\begin{proof}
  The \(G\)\nb-Hilbert spaces on which the two representations
  \(\xi_{Gp,\dots}\) and~\(\xi_{Gq,\dots}\) act are both
  isomorphic to \(\ind_H^G\chi\).  A homotopy from~\(p\)
  to~\(q\) inside~\(X^H\) provides a path of
  \(G\)\nb-equivariant \Star{}homomorphisms \(\CONT_0(X) \to
  \CONT_0(G/H) \to \Comp(\ind_H^G\chi)\).
\end{proof}

\begin{example}
  \label{exa:stretch_R_equivariantly}
  Let \(X\subseteq\C\) be the unit circle and let
  \(G=\Z/2=\{1,g\}\) act on~\(X\) by the conjugation \(z\mapsto
  \overline{z}\), that is, reflection at the real axis.  Let
  \(n\in\N_{\ge2}\) and consider the \(G\)\nb-equivariant
  self-map \(\selfmap(z)\defeq z^n\) of~\(X\).  Its fixed point
  set is the set of roots of unity of order \(n-1\).  The
  derivative \(D\selfmap\) is~\(n\) at all fixed points, so
  that \(\ID_\R-D\selfmap=1-n\) is a negative multiple of the
  identity map on~\(\R\).

  Let \(\tau\) and~\(\chi\) be the trivial and non-trivial
  representations of~\(G\) and let \(\C_\tau\) and~\(\C_\chi\)
  be~\(\C\) with the corresponding representation of~\(G\).  If
  \(p\in X\) is a \(\selfmap\)\nb-fixed point different
  from~\(\pm1\), then its stabiliser in~\(G\) is trivial, so
  that there is no representation~\(\chi_p\) to worry about.
  By Example~\ref{exa:orientation}, each of these fixed points
  contributes \(-[\xi_{\{z,\conj{z}\}}]\), where
  \(\xi_{\{z,\conj{z}\}}\) is the representation of
  \(\CONT_0(X)\) on \(\ell^2(\{z,\conj{z}\})\) by pointwise
  multiplication.  Example~\ref{exa:orientation_representation}
  shows that the contribution of~\(1\) is \(-[\xi_{\{1\},
    \chi}]\); this is just the class of the representation of
  \(\CONT_0(X)\) on~\(\C_\chi\) by evaluation at \(1\in X\).
  If~\(n\) is odd, then there is a similar contribution
  \(-[\xi_{\{-1\}, \chi}]\) from the fixed point~\(-1\).

  The limit of \(\xi_{\{z,\conj{z}\}}\) for \(z\to\pm1\) is
  \([\xi_{\{1\}, \chi\oplus \tau}]\) because \(\chi\oplus\tau\)
  is the regular representation of~\(G\), and by
  Lemma~\ref{pointevals}.  Hence the contributions of~\(\pm1\)
  are \([\xi_{\pm 1, \tau}] - [\xi_{\{z,\conj{z}\}}]\) for any
  \(z\in X\setminus\{\pm1\}\).  We may abbreviate \(\xi_{\pm1
    ,\tau}=\ev_{\pm1}\).  Thus we get
  \[
  \Lef(\selfmap) =
  \begin{cases}
    [\ev_1] + [\ev_{-1}] - \frac{n+1}{2}[\xi_{\{z,\conj{z}\}}]
    &\text{if \(n\) is odd,}\\
    [\ev_1] - \frac{n}{2}[\xi_{\{z,\conj{z}\}}]
    &\text{if \(n\) is even.}
  \end{cases}
  \]

  We may also triangulate~\(X\) by taking~\(\pm1\) as vertices
  and the upper and lower semi-circles as edges.  The
  combinatorial Lefschetz invariant of~\(\selfmap\) is given by
  the same formula in this case, as it should be.  When we
  forget the \(G\)\nb-action, \(\Lef(\selfmap)\) simplifies to
  \((1-n)\cdot [\ev]\); this is the expected result because
  \(1-n\) is the supertrace of the action of~\(\selfmap\) on
  the homology of~\(X\).
\end{example}

\begin{example}
  Let \(G\cong \Z\cross \Z/2\Z\) be the infinite dihedral
  group, identified with the group of affine transformations
  of~\(\R\) generated by \(u(x) = -x\) and \(w(x) = x+1\).
  Then~\(G\) has exactly two conjugacy classes of finite
  subgroups, each isomorphic to~\(\Z/2\).  Its action on~\(\R\)
  is proper, and the closed interval \([0,\nicefrac{1}{2}]\) is
  a fundamental domain.  There are two orbits of fixed point
  in~\(\R\) -- those of \(0\) and~\(\nicefrac{1}{2}\) -- and
  their stabilisers represent the two conjugacy classes of
  finite subgroups.

  Now we use some notation from
  Example~\ref{exa:orientation_representation}.  Each copy of
  \(\Z/2\) acting on the tangent space at the fixed point acts
  by multiplication by~\(-1\) on tangent vectors.  Therefore,
  the computations in
  Example~\ref{exa:orientation_representation} show that for
  any nonzero real number~\(A\), viewed as a linear
  transformation of the tangent space that commutes with
  \(\Z/2\), we have
  \[
  \sign (A) =
  \begin{cases}
    \R^\op_\chi&\text{if \(A <0\), and}\\
    \R_\tau&\text{if \(A>0\).}
  \end{cases}
  \]

  Let~\(\phi\) be a small \(G\)\nb-equivariant perturbation of
  the identity map \(\R\to\R\) with the following properties.
  First, \(\phi\) maps the interval \([0,\nicefrac{1}{2}]\) to
  itself.  Secondly, its fixed points in
  \([0,\nicefrac{1}{2}]\) are \(\nicefrac{1}{4}\) and the end
  points \(0\) and \(\nicefrac{1}{2}\); thirdly, its derivative is
  bigger than~\(1\) at both endpoints and between \(0\)
  and~\(1\) at \(\nicefrac{1}{4}\).  Such a map~\(\phi\)
  clearly exists.  Furthermore, it is homotopic to the identity
  map, so that \(\Lef(\phi)=\Eul_\R\).

  By construction, there are three fixed points modulo~\(G\),
  namely, the orbits of \(0\), \(\nicefrac{1}{4}\) and
  \(\nicefrac{1}{2}\).  The isotropy groups of the first and
  third orbit are non-conjugate subgroups isomorphic
  to~\(\Z/2\); from
  Example~\ref{exa:orientation_representation}, each of them
  contributes \(\R^\op_\chi\).  The point \(\nicefrac{1}{4} \)
  contributes the trivial character of the trivial subgroup.
  Hence
  \[
  \Lef (\phi) = - [\xi_{\Z,\chi}] - [\xi_{\Z+\nicefrac{1}{2},\chi}]
  + [\xi_{\Z+\nicefrac{1}{4}}].
  \]

  On the other hand, suppose we change the above map~\(\phi\)
  to fix the same points but to have zero derivative at~\(0\)
  and~\(\nicefrac{1}{2}\) and large derivative at
  \(\nicefrac{1}{4}\).  This is obviously possible.  Then we
  get contributions of~\(\R_\tau\) at~\(0\) and
  \(\nicefrac{1}{2}\) and a contribution of
  \(-[\xi_{\nicefrac{1}{4}}]\) at \(\nicefrac{1}{4}\).  Hence
  \[
  \Lef (\phi) = [\xi_{\Z, \tau}] +
  [\xi_{\Z+\nicefrac{1}{2},\tau}] - [\xi_{\Z+\nicefrac{1}{4}}].
  \]
  Combining both formulas yields
  \begin{equation}
    \label{mysteriousequality}
    [\xi_{\Z,\tau}] + [\xi_{\Z+\nicefrac{1}{2},\tau}]
    -[\xi_{\Z+\nicefrac{1}{4}}]
    = - [\xi_{\Z,\chi}] - [\xi_{\Z+\nicefrac{1}{2},\chi}]
    + [\xi_{\Z +\nicefrac{1}{4}}].
  \end{equation}
  By the way, the left-hand side is the description of
  \(\Eul_\R\) we get from the combinatorial dual with the
  obvious \(G\)\nb-invariant triangulation of~\(\R\) with
  vertex set \(\Z\cdot\nicefrac{1}{2} \subset \R\).

  It is possible, of course, to
  verify~\eqref{mysteriousequality} directly.  Indeed, since
  \(\tau + \chi\) is the regular representation~\(\lambda\),
  \eqref{mysteriousequality} is equivalent to
  \[
  [\xi_{\Z,\lambda}] + [\xi_{\Z+\nicefrac{1}{2},\lambda}] =
  2[\xi_{\Z+\nicefrac{1}{4}}],
  \]
  and this relation follows from Lemma~\ref{pointevals}.
\end{example}

Now let~\(G\) be a Lie group acting properly and isometrically
on a Riemannian manifold~\(X\).  Assume that \(G\backslash X\)
is compact.  Let \(\selfmap\colon X \to X\) be a smooth map.
Assume that \(F \defeq \Fix(\selfmap)\) is a manifold and the
restriction of \(\ID - D_p\psi\) to the normal bundle of
\(T_pF\subseteq T_pX\) is injective for all \(p\in F\).
Restriction of the Riemannian metric on~\(X\) to~\(F\) gives a
Riemannian metric on~\(F\).  Since~\(F\) is closed and
\(G\backslash X\) is compact, \(G\backslash F\) is compact.
Hence \(G\backslash F\) is complete with respect to the
restricted metric.

Let~\(H\) be a Lie group acting isometrically and by
orientation-preserving maps on an odd-dimensional complete
Riemannian manifold~\(F\).  Then \(\Eul_F \cong 0\)
(see~\cite{Emerson-Meyer:Euler}).  Using
Theorem~\ref{the:Lef_smooth}, we see that odd-dimensional
components of \(G\backslash F\) do not contribute to
\(\Lef(\selfmap)\).  Only even-dimensional components may have
non-trivial Euler characteristics.

\begin{example}
  \label{liegroupexample}
  Let \(\psi\colon \Sphere^3 \to \Sphere^3\) be the reflection
  \(\psi(x,y,z,w) \defeq (x,y,z,-w)\), where
  \[
  \Sphere^3 \defeq
  \{(x,y,z,w)\in \R^4 \mid x^2+y^2+z^2+w^2 =1\}.
  \]
  The group \(\textup{O}(3,\R) \subset \textup{O}(4,\R)\) acts
  on~\(\Sphere^3\) commuting with~\(\psi\), so that
  Theorem~\ref{the:Lef_smooth} applies.  We have \(\Fix(\psi)
  \cong \Sphere^2\) and
  \[
  \Lef(\psi) = - i_*(\Eul_{\Sphere^2})
  \in \KK^{\textup{O}(3,\R)}_0(\CONT(\Sphere^3),\C),
  \]
  where \(i\colon \Sphere^2 \to \Sphere^3\) is the inclusion as
  the equator.  The bundle \(\sign (\ID - D\psi)\) is the
  trivial line bundle with trivial action of
  \(\textup{O}(3,\R)\), with negative parity; this accounts for
  the sign.  Note that \(\Eul_{\Sphere^2}\) with respect to the
  group \(\textup{O}(3,\R)\) is nonzero because the
  (non-equivariant) Euler characteristic of~\(\Sphere^2\)
  is~\(2\).  The equivariant index of \(\Eul_{\Sphere^2}\) in
  \(\KK^{\textup{O}(3,\R)}_0(\C,\C)\) is given by the character
  \(\tau + \det\), where~\(\tau\) is the trivial character and
  \(\det\colon \textup{O}(3,\R) \to \{-1,+1\}\) is the
  determinant.  This follows from the usual description of the
  index of the de Rham operator in terms of spaces of harmonic
  forms.
\end{example}

Finally, at least in the non-equivariant case, we record that
maps of non-compact \(G\)\nb-spaces usually have zero
Lefschetz invariants:

\begin{corollary}
  \label{cor:Lef_vanishes}
  In the situation of Theorem~\textup{\ref{the:Lef_smooth}},
  assume that~\(G\) is the trivial group, \(\selfmap\)
  and~\(\ID_X\) are transverse, and none of the connected
  components of~\(X\) are compact.  Then \(\Lef(\selfmap)=0\).
\end{corollary}

\begin{proof}
  There is an increasing sequence of compact subsets~\((K_m)\)
  in~\(X\) such that, for each \(m\in\N\), no component of
  \(X\setminus K_m\) is compact.  For each \(p\in Y\),
  let~\(m\) be minimal with \(p\in K_m\), and let~\(I_p\) be a
  path from~\(p\) to~\(\infty\) in~\(X^+\) that does not meet
  \(K_{m-1}\).  This ensures that \(\bigoplus_{p\in Y}
  \ev_{I_p(t)}\) is a homotopy from the \Star{}homomorphism
  describing \(\Lef(\selfmap)\) to the zero map.
\end{proof}

The \emph{equivariant} Lefschetz map can carry information also
for non-compact spaces.  An analogue of
Corollary~\ref{cor:Lef_vanishes} holds in the equivariant case
for discrete~\(G\) provided for each finite subgroup
\(H\subseteq G\), the fixed point submanifold \(\{x\in X\mid
Hx=x\}\) has no compact components; it suffices to assume this
only for those components that contain a fixed point
of~\(\selfmap\).

\subsection{The Lefschetz formula: comparing the two computations}
\label{sec:Lef_formula}

\begin{theorem}
  \label{the:Lef_Fixed_Point}
  Let~\(G\) be a discrete group and let~\(X\) be a smooth
  manifold equipped with a proper action of~\(G\).  Let
  \(\selfmap\colon X\to X\) be a smooth map that satisfies the
  conditions of Theorem \textup{\ref{the:Lef_smooth}}.  Choose
  a \(G\)\nb-equivariant cellular decomposition of~\(X\), and
  let \(\selfmap'\colon X\to X\) be \(G\)\nb-homotopic
  to~\(\selfmap\) and cellular.  Then
  \begin{multline*}
    \sum_{\dot\sigma\in G\backslash SX}
    \mult(\Selfmap,\dot\sigma)[\xi_{\dot\sigma}]
    = \Lef(\selfmap') \\= \Lef(\selfmap) =
    r_Y \otimes_{\CONT_0(Y)} \sign(\ID_\nu- D_\nu\selfmap)
    \otimes_{\CONT_0(Y)} \Eul_Y.
  \end{multline*}
  Here \(SX\) denotes the set of cells in the decomposition,
  and the multiplicities \(\mult(\Selfmap,\dot\sigma)\) and the
  equivariant \(\K\)\nb-homology classes \([\xi_{\dot\sigma}]\)
  are as in Theorem~\textup{\ref{the:Lef_combinatorial}};
  \(Y=\Fix(\selfmap)\), \(r_Y\)~is the class of the restriction
  map \(\CONT_0(X)\to\CONT_0(Y)\), \(\nu\)~is the normal bundle
  of~\(Y\), and \(\Eul_Y\) is the equivariant
  \(\K\)\nb-homology class of the de Rham operator on~\(Y\) as
  in Theorem~\textup{\ref{the:Lef_smooth}}.
\end{theorem}

\begin{proof}
  Since~\(G\) is discrete and acts properly, \(X\) carries a
  \(G\)\nb-equivariant triangulation
  (see~\cite{Illman:Equivariant_triangulations}).  For the same
  reasons, \(X\) admits a complete \(G\)\nb-invariant
  Riemannian metric.  The combinatorial Lefschetz invariant
  \(\sum_{\dot\sigma\in G\backslash SX}
  \mult(\Selfmap,\dot\sigma) [\xi_{\dot\sigma}]\) is
  independent of the cellular decomposition.  Hence we may
  compute it using the cellular structure underlying a
  triangulation (our duality approach does not work for general
  \(G\)\nb-CW-complexes).  Any \(G\)\nb-equivariant self-map of
  a \(G\)\nb-CW-complex is \(G\)\nb-equivariantly homotopic to
  a cellular \(G\)\nb-equivariant self-map, so
  that~\(\selfmap'\) exists.  We have
  \(\Lef(\selfmap')=\Lef(\selfmap)\) because the Lefschetz
  invariant is homotopy invariant.  Now combine the formulas in
  Theorems \ref{the:Lef_combinatorial}
  and~\ref{the:Lef_smooth}.
\end{proof}

This result is similar to the Lefschetz Fixed Point Formula
of~\cite{Lueck-Rosenberg:Lefschetz}.  There are two
differences: first, we allow~\(\selfmap\) to have non-isolated
fixed points and describe the local contributions of fixed
points differently even in the isolated case; secondly, we
compute an element of \(\KK^G_0(\CONT_0(X),\UNIT)\) instead of
the universal equivariant homology theory considered
in~\cite{Lueck-Rosenberg:Lefschetz}.

\section{Lefschetz invariants for simplicial complexes}
\label{sec:combinatorial_proof}

We first recall briefly the combinatorial Kasparov dual
in~\cite{Emerson-Meyer:Euler}.  This construction is suggested
by work of Gennadi Kasparov and Georges Skandalis
in~\cite{Kasparov-Skandalis:Buildings}.  In order to compute
Euler characteristics in~\cite{Emerson-Meyer:Euler} we
reconstructed the results of Kasparov and Skandalis in more
explicit terms.  We refer the reader to our previous article
for details and some notation.  We are going to use the
combinatorial Kasparov dual to compute the Lefschetz invariant
of a cellular self-map.  This is a matter of plugging all the
ingredients into~\eqref{eq:compute_Lef} and simplifying the
outcome.

\subsection{Description of the combinatorial dual}
\label{sec:combinatorial_dual}

Let~\(X\) be a simplicial complex and let~\(G\) be a locally
compact group acting simplicially on~\(X\).  Thus~\(G\)
permutes the simplices of~\(X\), and the stabiliser of each
simplex is an open subgroup of~\(G\).  Since this forces the
connected component of the identity in~\(G\) to act trivially,
the group~\(G\) will usually be totally disconnected, but we do
not need this assumption.  Let \(S_dX\) be the set of
\(d\)\nb-simplices in~\(X\) and \(SX=\bigsqcup S_dX\).  In
particular, \(S_0X\) is the set of vertices of~\(X\).

We assume throughout that~\(X\) is finite-dimensional, say of
dimension at most~\(n\).  Let \(\no \defeq \{1,\dotsc,n\}\).
An (\(n\)\nb-dimensional, \(G\)\nb-invariant)
\emph{colouring} on~\(X\) is a map \(\col\colon S_0X \to \no\)
that is \(G\)\nb-invariant and satisfies
\(\col(\sigma)\neq\col(\tau)\) whenever \(\sigma,\tau\in S_0X\)
are joined by an edge.  Equivalently, if \(\sigma\in SX\) is
any simplex, then~\(\col\) restricts to an injective map on the
set of vertices of~\(\sigma\).  Coloured simplicial complexes
are also called \emph{typed}.

The barycentric subdivision of an \(n\)\nb-dimensional
simplicial complex admits a canonical \(n\)\nb-dimensional
colouring, so that we may assume without loss of generality
that~\(X\) itself has such a colouring~\(\col\).
Since~\(\col\) is \(G\)\nb-invariant, a group element that
maps a simplex to itself must fix it pointwise.

We use the \(n\)\nb-dimensional affine space
\[
E \defeq \{ (t_0,\dotsc,t_n)\in\R^{n+1} \mid
t_0+\dotsb+t_n = 1\}
\]
and the standard \(n\)\nb-simplex
\[
\Sigma \defeq \{ (t_0,\dotsc,t_n)\in E \mid
\text{\(t_i\ge0\) for \(i=0,\dotsc,n\)}\}.
\]

For a non-empty subset \(f\subseteq\no\), we let
\(\Sigma_f\subseteq\Sigma\) be the corresponding face; its
points are characterised by \(t_j=0\) for \(j\notin f\).  This
subset is denoted by~\(\abs{f}\) in~\cite{Emerson-Meyer:Euler}.
The map \(f\mapsto \Sigma_f\) gives a bijection between the set
\(\poset\) of non-empty subsets of~\(\no\) and the set of faces
of~\(\Sigma\).

Let~\(\tau\) be a simplex in~\(X\) with vertices
\(v_0,\dotsc,v_d\), where \(d=\dim \tau\).  We also
write~\(\abs{\tau}\) for the corresponding subset of~\(X\).
Points in~\(\abs{\tau}\) can be described by barycentric
coordinates as \(t_0v_0+\dotsb+t_dv_d\) with
\(t_0,\dotsc,t_d\in\R_{\ge0}\) and \(t_0+\dotsb+t_d=1\).  Since
the colouring~\(\col\) is injective on the vertices
of~\(\tau\), the set
\[
\col(\tau) \defeq \{\col(v_0),\dotsc,\col(v_d)\}
\]
has \(d+1\) elements and hence determines a
\(d\)\nb-dimensional face \(\Sigma_{\col(\tau)}\)
of~\(\Sigma\).  Even more, \(\col\) induces a linear bijection
\[
\abs{\col}\colon \abs{\tau} \to \Sigma_{\col(\tau)},
\qquad t_0v_0+\dotsb+t_dv_d \mapsto
t_0 e_{\col(v_0)} + \dotsb +  t_d e_{\col(v_d)}.
\]
Here \(e_i\in\Sigma\subseteq E\) is the \(i\)th standard basis
vector.

We define
\begin{align*}
  R_f &\defeq \{t\in E \mid
  \text{\(t_j\ge0\) for \(j\in f\) and
    \(t_j\le 0\) for \(j\notin f\)}\}.\\
  R_{\le f} &\defeq \bigcup_{l\subseteq f} R_l =
  \{t\in E \mid
  \text{\(t_j\le 0\) for \(j\in \no\setminus f\)}\}.
\end{align*}

We introduce \emph{homogeneous coordinates} by letting
\[
[t_0:\dotsc:t_n] \defeq \frac{(t_0,\dotsc,t_n)}{t_0+\dotsb+t_n}
\qquad\text{for \((t_0,\dotsc,t_n)\in\R^{n+1}\) with
\(t_0+\dotsb+t_n\neq0\).}
\]
We define a retraction \(q\colon E\to\Sigma\) from~\(E\)
to~\(\Sigma\) by
\[
q(t) \defeq \bigl[\max(t_0,0):\dotsc:\max(t_n,0)\bigr]
\qquad
\text{for \(t= (t_0,\dotsc,t_n) \in E\).}
\]
Thus \(R_{\le f} = q^{-1}(\Sigma_f)\) for \(f\in\poset\).

Let \(\Comp \defeq \Comp(\ell^2SX)\).  If \(T\in\Comp\), let
\(T_{\sigma,\sigma'}\) for \(\sigma, \sigma'\in SX\) be its
matrix coefficients.

\begin{definition}
  \label{def:combinatorial_dual_algebra}
  The \emph{dual \Cstar{}algebra} of~\(X\) is defined by
  \[
  \dual \defeq \{\varphi \in \CONT_0(E,\Comp) \mid
  \text{\(\supp \varphi_{\sigma,\sigma'} \subseteq
    R_{\le\col(\sigma\cap\sigma')}\) for all
    \(\sigma,\sigma'\in SX\)}\}.
  \]
\end{definition}

\begin{example}
  \label{exa:combinatorial_dual_interval}
  Suppose~\(X\) is a single \(1\)\nb-simplex -- together with
  its two vertices.  Thus \(SX\) has three elements.  The dual
  \Cstar{}algebra is the algebra of \(3\times3\)-matrices with
  a pattern of entries of the form
  \[
  \begin{pmatrix}
    \CONT_0(\R_{<0})& \CONT_0(\R_{<0})& 0\\
    \CONT_0(\R_{<0})& \CONT_0(\R)& \CONT_0(\R_{>1})\\
    0& \CONT_0(\R_{>1})& \CONT_0(\R_{>1})
  \end{pmatrix}.
  \]
  This \Cstar{}algebra is Morita--Rieffel equivalent to
  \(\CONT_0(\R)\).  This is to be expected because \([0,1]\) is
  homotopy equivalent to the point, so that its dual should be,
  up to dimension shift, \(\KK\)-equivalent to~\(\UNIT\).
\end{example}

We return to the general case.  The pointwise multiplication
homomorphism \(\CONT_0(E)\otimes \dual\to\dual\)
turns~\(\dual\) into a \Cstar{}algebra over~\(E\).  To describe
the fibre at \(t\in E\), let \(f \defeq \{j\in\no\mid t_j\ge
0\}\); equivalently, this is the minimal subset of~\(\no\) for
which~\(t\) is contained in the interior of \(R_{\le f}\).  Let
\(SX_{\ge f} \defeq \{\sigma\in SX\mid \col(\sigma)\supseteq
f\}\) and define a relation on \(SX_{\ge f}\) by
\[
\sigma\sim_f \sigma'
\quad\iff\quad \col(\sigma\cap\sigma')\supseteq f.
\]
This is an equivalence relation because~\(\col\) is a
colouring.  The fibre of~\(\dual\) over~\(t\) is the
\Cstar{}algebra of the equivalence relation~\(\sim_f\) on
\(SX_{\ge f}\) or, equivalently,
\[
\Comp_f \defeq \{\varphi\in\Comp\mid
\text{\(\varphi_{\sigma,\sigma'}=0\) unless
  \(\col(\sigma\cap\sigma')\supseteq f\)}\}.
\]
This is a direct sum of matrix algebras with one summand for
each simplex in~\(X\) of colour~\(f\).

To be part of a Kasparov dual, \(\dual\) must be a
\Cstar{}algebra over~\(X\).  This structure can also be
described by a non-degenerate \(G\)\nb-equivariant
\Star{}homomorphism \(m\colon \CONT_0(X,\dual) \cong
\CONT_0(X)\otimes\dual\to\dual\).  We recall its description in
\cite{Emerson-Meyer:Euler}*{Equation (42)}:
\[
m(\varphi)_{\sigma\sigma'}(t) =
\varphi_{\sigma\sigma'}(\bar{q}_\sigma(t),t)
\]
for all \(\varphi\in \CONT_0(X,\dual)\), \(t\in E\), and,
\(\sigma,\sigma'\in SX\).  This involves the map
\[
\bar{q}\colon E\times SX\to X,
\qquad (t,\sigma)\mapsto \bar{q}_\sigma(t)
\]
which is defined as follows.  If
\(\varphi_{\sigma\sigma'}(\blank,t)\colon X\to\C\) is not
identically zero, then~\(t\) must lie in the interior of
\(R_{\le \col(\sigma\cap\sigma')} \subseteq R_{\le
  \col(\sigma)}\), so that we only need \(\bar{q}_\sigma(t)\)
for \(t\in R_{\le \col(\sigma)}\).  Then \(q(t) \in
\Sigma_{\col(\sigma)}\), which is identified with
\(\abs{\sigma}\subseteq X\) by the colouring.  We let
\(\bar{q}_\sigma(t)\) be the point of~\(\abs{\sigma}\) that
corresponds to \(q(t)\).

To compute the Lefschetz invariant of a self-map~\(\selfmap\),
we must combine the multiplication map \(m\colon
\CONT_0(X,\dual)\to\dual\) with~\(\selfmap\) (see
Lemma~\ref{lem:compute_T_on_map}).  The resulting
\Star{}homomorphism \(\mu_\selfmap\colon
\CONT_0(X)\otimes\dual\to\dual\) in~\eqref{eq:compute_Lef} is
given by
\begin{equation}
  \label{eq:combinatorial_mu_selfmap}
  \mu_\selfmap(\varphi)_{\sigma,\sigma'} (t) =
  \varphi \bigl(\selfmap\circ
  \bar{q}_\sigma(t), t\bigr)_{\sigma,\sigma'}.
\end{equation}

Next we describe the class \(D\in\KK^G_n(\dual,\UNIT)\) for the
Kasparov dual.  Let
\[
\beta_E \in \KK^G_n(\CONT_0(E), \UNIT)
\quad \text{and}\quad
\widehat{\beta}_E = \beta_E^{-1} \in
\KK^G_{-n}\bigl(\UNIT,\CONT_0(E)\bigr)
\]
be the Bott class and its inverse and let \(i\colon \dual \to
\CONT_0(E)\otimes \Comp\) be the obvious inclusion.  The latter
defines a class \([i]\in\KK^G_0\bigl(\dual,\CONT_0(E)\bigr)\)
because \(\KK^G\) is stable.  We let
\[
D \defeq [i]\otimes_{\CONT_0(E)} \beta_E \in
\KK^G_n(\dual,\UNIT).
\]

The final datum~\(\Theta\) is more involved and uses two
ingredients:
\begin{enumerate}
\item a \(G\)\nb-equivariant continuous map~\(v\) from~\(X\)
  to the space of unit vectors in \(\ell^2(SX)\), and

\item a family of \Star{}homomorphisms \(h_s!\colon
  \CONT_0(E)\to\CONT_0(E)\) for \(s\in E\).
\end{enumerate}

The map~\(v\) has two main features: first, the support of
\(v(x) \in \ell^2(SX)\) is the set of faces of~\(\sigma\) if
\(x\in\abs{\sigma}\); secondly, \(v(x)=\delta_\sigma\) if
\(x\in\abs{\sigma}\) and \(\abs{\col}(x)\in
CR_{\col(\sigma)}\), where the regions \(CR_f\) are defined by
\[
CR_f \defeq \{t\in\Sigma \mid
\text{\(t_i \ge L\) for \(i\in f\) and \(t_i \le L\)
for \(i\in\no\setminus f\)}\}
\]
with an auxiliary parameter \(L>0\).  The subsets \(CR_f\) for
\(f\in\poset\) cover~\(\Sigma\).  Each \(CR_f\) is a closed
polygonal neighbourhood of the barycentre of~\(\Sigma_f\)
in~\(\Sigma\).  The shape of these regions for \(n=2\) is
indicated in \cite{Emerson-Meyer:Euler}*{Figure~3}.

Passing from unit vectors to rank-\(1\)-projections, we get a
map
\[
P\colon X\to \Comp,
\qquad x\mapsto
\bigl\lvert v(x)\bigr\rangle \bigl\langle v(x)\bigr\rvert.
\]
Let \(x\in\abs{\sigma}\).  The properties of~\(v\) ensure two
things.  First, \(P(x)_{\tau,\tau'}=0\) unless both \(\tau\)
and~\(\tau'\) are faces of~\(\sigma\).  Secondly, if
\(\abs{\col}(x)\in CR_{\col(\sigma)}\), then
\(P(x)_{\tau,\tau'}=0\) unless \(\tau=\tau'=\sigma\), and
\(P(x)_{\sigma,\sigma}=1\).

The second ingredient for~\(\Theta\) is a family of
\Star{}homomorphisms
\[
h_s!\colon \CONT_0(E) \to \CONT_0(E)
\]
for \(s\in E\), which is constructed as follows.  Let
\[
B(\delta) \defeq \{t\in\R^{n+1}\mid
\text{\(t_0+\dotsb+t_n = 0\) and \(\abs{t_i}<\delta\) for
  \(j=0,\dotsc,n\)}\}.
\]
Fix \(\lambda > \frac{1}{1-(n+1)L}\) and let \(r_\lambda\colon
E\to E\) be the radial expansion by a factor of~\(\lambda\)
around the barycentre of~\(\Sigma\):
\[
r_\lambda(t_0,\dotsc,t_n) \defeq
\left(\lambda t_0 - \frac{\lambda-1}{n+1},\dotsc,
\lambda t_n - \frac{\lambda -1}{n+1}\right).
\]
By \cite{Emerson-Meyer:Euler}*{Lemma 27}, we can find
\(\delta>0\) such that
\begin{equation}
  \label{choiceofdelta}
  r_\lambda (s) + B(\delta) \subseteq R_{\le f}
  \qquad\text{for all \(f\in\poset\) and \(s\in CR_f\).}
\end{equation}

Let \(h\colon E \xrightarrow{\cong} B(\delta)\) be a fixed,
orientation-preserving diffeomorphism.  For \(s\in E\), we
define
\begin{equation}
  \label{eq:def_h_shriek}
  h_s!\colon \CONT_0(E)\to\CONT_0(E),
  \qquad
  h_s!(\varphi)(t) \defeq
  \begin{cases}
    \varphi\bigl(h^{-1}(t-s)\bigr)
    &\text{for \(t-s\in B(\delta)\),}\\
    0
    &\text{otherwise.}\\
  \end{cases}
\end{equation}
The definitions above ensure that
\begin{equation}
  \label{eq:def_theta}
  \vartheta_x(\varphi) \defeq
  h_{\abs{\col}(x)}!(\varphi)\otimes P(x)
\end{equation}
is a \Star{}homomorphism from \(\CONT_0(E)\) to~\(\dual\) for
all \(x\in X\).  Letting~\(x\) vary, we get a \(G\ltimes
X\)-equivariant \Star{}homomorphism
\[
\vartheta\colon \CONT_0(X\times E) \cong
\CONT_0(X)\otimes \CONT_0(E) \to
\CONT_0(X)\otimes\dual \cong \CONT_0(X,\dual).
\]
This yields \([\vartheta] \in \RKK^G_0(X; \CONT_0(E),\dual)\).

Finally, \(\Theta\in\RKK^G_{-n}(X;\UNIT,\dual)\) is defined by
\[
\Theta \defeq \widehat{\beta}_E \otimes_{\CONT_0(E)} [\vartheta],
\]
where \(\widehat{\beta}_E \in \KK^G_{-n}(\UNIT,\CONT_0(E))\) is
the generator of Bott periodicity.

The data \((\dual,\Theta,D)\) defined above is a
\(G\)\nb-equivariant Kasparov dual for~\(X\) of dimension
\(-n\) by \cite{Emerson-Meyer:Euler}*{Theorem 29}.

\subsection{Computing the Lefschetz invariant}
\label{sec:combinatorial_compute_Lefschetz}

We now compute the Lefschetz invariant of a
\(G\)\nb-equivariant self-map \(\selfmap\colon X\to X\) using
the simplicial dual described above.  The map~\(\selfmap\) is
\(G\)\nb-equivariantly homotopic to a \(G\)\nb-equivariant
cellular map, that is, a map that preserves the filtration
of~\(X\) by skeleta.  Even more, by an equivariant version of
the Simplicial Approximation Theorem, any self-map is
\(G\)\nb-equivariantly homotopic to a \(G\)\nb-equivariant
simplicial map \(\selfmap'\colon \Sd X \to X\), where \(\Sd X\)
denotes a sufficiently fine \(G\)\nb-invariant subdivision
of~\(X\) (since we allow the subdivision to get finer near
infinity, this even works if~\(G\) does not act cocompactly
on~\(X\)).  Since the Lefschetz invariant of a self-map is
homotopy invariant, we may assume that~\(\selfmap\) has this
special form from now on.

The map~\(\selfmap\) is a cellular map with respect to the
cellular decomposition underlying the original simplicial
structure on~\(X\), so that it induces a cellular chain map
\(\Selfmap\colon C_\bullet(X)\to C_\bullet(X)\).  The matrix
coefficient \(\Selfmap_{\sigma\tau}\) of~\(\Selfmap\) for two
simplices \(\sigma\),~\(\tau\) counts how many of the simplices
in the subdivision of~\(\tau\) are mapped onto~\(\sigma\), with
a sign depending on whether the map
\(\selfmap|_{\abs{\tau}}\colon \abs{\tau}\to\abs{\sigma}\)
preserves or reverses orientation.

The starting point for our computation
is~\eqref{eq:compute_Lef}.  We must plug in the ingredients we
have constructed above.  This yields
\begin{multline}
  \label{eq:combinatorial_Lef_1}
  \Lef(\selfmap) = \overline{\Theta} \otimes_{\CONT_0(X,\dual)}
  [\mu_\selfmap]\otimes_\dual D
  \\= \widehat\beta_E \otimes_{\CONT_0(E)} [\vartheta]
  \otimes_{\CONT_0(X,\dual)} [\mu_\selfmap]
  \otimes_\dual [i] \otimes_{\CONT_0(E)} \beta_E
\end{multline}
in \(\KK^G_0(\CONT_0(X), \UNIT)\).  Here \(\beta_E\)
and~\(\widehat\beta_E\) are the Bott class and its inverse, and
the remaining ingredients are given by \Star{}homomorphisms
\[
\CONT_0(X)\otimes\CONT_0(E) \xrightarrow{\vartheta}
\CONT_0(X,\dual) \xrightarrow{\mu_\selfmap}
\dual \xrightarrow{i}
\CONT_0(E,\Comp).
\]
The maps \(\vartheta\) and~\(\mu_\selfmap\) are described in
\eqref{eq:def_theta} and~\eqref{eq:combinatorial_mu_selfmap},
the map~\(i\) is just the embedding.  The Kasparov product for
\Star{}homomorphisms agrees with the usual composition.  Hence
\begin{equation}
  \label{eq:combinatorial_Lef_2}
  \Lef(\selfmap) =
  \widehat\beta_E \otimes_{\CONT_0(E)}
  [\Xi_\selfmap]
  \otimes_{\CONT_0(E)} \beta_E
\end{equation}
with the \(G\)\nb-equivariant \Star{}homomorphism
\[
\Xi_\selfmap \defeq i\circ\mu_\selfmap\circ\vartheta\colon
\CONT_0(X\times E)\to\CONT_0(E)\otimes\Comp(\ell^2SX).
\]
Equations \eqref{eq:def_theta}
and~\eqref{eq:combinatorial_mu_selfmap} yield
\begin{equation}
  \label{eq:xi}
  \Xi_\selfmap(\varphi)_{\sigma,\sigma'}(t) =
  \varphi \Bigl( \selfmap\circ\bar{q}_\sigma(t),
  h^{-1}\bigl(t-r_\lambda \circ \abs{\col} \circ \selfmap \circ
  \bar{q}_\sigma(t)\bigr)\Bigr)
  \cdot P_{\sigma,\sigma'} \bigl(\selfmap\circ\bar{q}_\sigma(t)\bigr)
\end{equation}
for all \(\sigma,\sigma'\in SX\), \(t\in E\), \(\varphi\in
\CONT_0(X\times E)\); this is understood to be~\(0\) unless
\(t-r_\lambda \circ \abs{\col} \circ \selfmap \circ
\bar{q}_\sigma(t) \in B(\delta)\).

\begin{lemma}
  \label{lem:simplices_in_xi}
  Let \(\varphi \in \CONT_0(X\times E)\), \(t\in E\), and
  \(\sigma,\sigma'\in SX\), and let \(x\defeq \selfmap\circ
  \bar{q}_\sigma(t)\).  Assume that
  \(\Xi_\selfmap(\varphi)_{\sigma,\sigma'}(t) \neq 0\).  Then:
  \begin{enumerate}
  \item \(t\in B(\delta) + r_\lambda \circ \abs{\col}(x)\),

  \item \(\sigma = \sigma'\),

  \item \(\abs{\col}(x) \in \Sigma_{\col(\sigma)} \cap
    CR_{\col(\sigma)}\),

  \item \(t\) belongs to the interior of \(R_{\col(\sigma)}\).

  \end{enumerate}
\end{lemma}

\begin{proof}
  The first claim follows immediately from the definition
  of~\(h_s!\) in~\eqref{eq:def_h_shriek}.  The next two
  properties are more interesting.

  Choose \(f\subseteq\no\) with \(\abs{\col}(x)\in CR_f\).  We
  must have \(t\in B(\delta) +
  r_\lambda\bigl(\abs{\col}(x)\bigr) \subseteq R_{\le f}\) by
  \cite{Emerson-Meyer:Euler}*{Lemma 27}.  Thus \(q(t)\in
  \Sigma_f\) and hence \(\bar{q}_\sigma(t)\) belongs to the
  \(d\)\nb-skeleton of~\(X\) with \(d\defeq\dim f =
  \abs{f}-1\).

  Since~\(\selfmap\) is cellular, \(x=
  \selfmap\bigl(\bar{q}_\sigma(t)\bigr)\) also belongs to the
  \(d\)\nb-skeleton.  Hence the point \(s\defeq \abs{\col}(x)
  \in \Sigma\) belongs to some \(d\)\nb-dimensional face
  of~\(\Sigma\).  Thus at most \(d+1\) of its coordinates
  \(s_0,\dotsc,s_n\) can be non-zero.  But \(s\in CR_f\) means
  that \(s_j\ge L\) for \(j\in f\), providing \(d+1\) non-zero
  coordinates.  Hence \(s_j\ge L\) for \(j\in f\) and \(s_j=0\)
  for \(j\in \no\setminus f\) or, equivalently, \(s\in
  \Sigma_f\cap CR_f\).  But then \(P(x)\) is the projection
  onto the basis vector~\(\delta_\tau\), where~\(\tau\) is the
  unique simplex in~\(X\) with \(x\in \tau\) and
  \(\col(\tau)=f\).  Thus \(\tau= \sigma=\sigma'\) and
  \(f=\col(\sigma)\).  We also get \(\abs{\col}(x) =
  s\in\Sigma_{\col(\sigma)} \cap CR_{\col(\sigma)}\) as
  asserted in (3).  We have already seen above that
  \(t=(t_0,\dotsc,t_n)\) must belong to
  \(R_{\le\col(\sigma)}\), that is, \(t_i\le0\) for
  \(i\in\no\setminus f\).  If \(t_i\le0\) for some \(i\in f\),
  then \(\bar{q}_\sigma(t)\) belongs to the \(d-1\)-skeleton
  of~\(X\), which is impossible.  Hence \(t_i>0\) for all
  \(i\in f\).  Furthermore, \(t\in B(\delta) +
  r_\lambda(\Sigma_f)\) implies \(t_i<0\) for \(i\in
  \no\setminus f\).  Thus~\(t\) is an interior point
  of~\(R_f\).
\end{proof}

As a result, the range of~\(\Xi_\selfmap\) is contained in the
\Cstar{}subalgebra
\[
\CONT_0(E\times SX) \cong
\CONT_0(E)\otimes \CONT_0(SX)\subseteq
\CONT_0(E)\otimes \Comp(\ell^2SX)
\]
of all operators that are diagonal on \(\ell^2(SX)\) in the
standard basis.  Let
\[
\Xi_{\selfmap,\sigma}\colon \CONT_0(X\times E)\to \CONT_0(E)
\]
be the value at \(\sigma\in SX\), so that \(\Xi_\selfmap =
(\Xi_{\selfmap,\sigma})_{\sigma\in SX}\).

Given \(t\in E\) and \(\sigma\in SX\), we define \(y\defeq
\bar{q}_\sigma(t)\) and \(x\defeq \selfmap(y) =
\selfmap\circ\bar{q}_\sigma(t)\) and let \(f\defeq
\col(\sigma)\).  Lemma~\ref{lem:simplices_in_xi} yields
\(\Xi_{\selfmap,\sigma}(\varphi)(t) = 0\) unless
\(x\in\abs{\sigma}\) and \(\abs{\col}(x) \in CR_f\).  If
\(x\in\abs{\sigma}\) and \(\abs{\col}(x) \in CR_f\), then
\[
\Xi_{\selfmap,\sigma}(\varphi)(t) =
\varphi \Bigl(x, h^{-1}\bigl(t-r_\lambda \circ
\abs{\col}(x)\bigr)\Bigr)
\]
because \(v(x) = \delta_\sigma\).  We have \(y\in\abs{\sigma}\)
because the range of~\(\bar{q}_\sigma\) is contained
in~\(\sigma\).  Thus \(\Xi_{\selfmap,\sigma}(\varphi)\)
vanishes identically unless there is \(y\in\abs{\sigma}\) with
\(\selfmap(y)\in\abs{\sigma}\).

If \(x\in\abs{\sigma}\), then there is a canonical linear
homotopy from~\(x\) to the barycentre~\(\xi_\sigma\).
Deforming \(\Xi_{\selfmap,\sigma}\) along this homotopy, we get
the map \(\Xi'_{\selfmap,\sigma}\colon \CONT_0(X\times E)\to
\CONT_0(E)\) defined by
\[
\Xi'_{\selfmap,\sigma}(\varphi)(t) =
\varphi \Bigl(\xi_\sigma, h^{-1}\bigl(t-r_\lambda \circ
\abs{\col}(x)\bigr)\Bigr)
\]
for \(x\in\abs{\sigma}\) and \(\abs{\col}(x) \in CR_f\),
and~\(0\) otherwise.  Hence the families of maps
\((\Xi_{\selfmap,\sigma})_{\sigma\in SX}\) and
\((\Xi'_{\selfmap,\sigma})_{\sigma\in SX}\) define the same
class in \(\KK^G_0\bigl(\CONT_0(X\times E),\CONT_0(E)\bigr)\).

The map \(\Xi'_{\selfmap,\sigma}\) is an exterior product of
the evaluation map
\[
\CONT_0(X)\to\UNIT,
\qquad\varphi\mapsto\varphi(\xi_\sigma),
\]
and a certain endomorphism \(\Xi''_{\selfmap,\sigma}\) of
\(\CONT_0(E)\).  Since~\(\selfmap\) is \(G\)\nb-equivariant,
\(\Xi''_{\selfmap,\sigma}\) only depends on the orbit
\(\dot\sigma\defeq G\sigma\) of~\(\sigma\).  Hence we get
\[
\Lef(\selfmap) =
\sum_{\dot\sigma \in G\backslash SX} [\xi_{\dot\sigma}]\otimes
(\beta_E\otimes_{\CONT_0(E)} [\Xi''_{\selfmap,\dot\sigma}]
\otimes_{\CONT_0(E)} \widehat{\beta}_E)
\]
with \([\xi_{\dot\sigma}]\) as defined in~\eqref{eq:dot_xi}.
Since \(\KK_0\bigl(\CONT_0(E),\CONT_0(E)\bigr) \cong \Z\) by
Bott periodicity, the class \([\Xi''_{\selfmap,\dot\sigma}]\)
cannot contribute more than a multiplicity.  This is determined
by the following lemma:

\begin{lemma}
  \label{lem:combinatorial_mult}
  \([\Xi''_{\selfmap,\dot\sigma}] =
  \mult(\Selfmap,\dot\sigma)\cdot [\ID]\) in
  \(\KK_0\bigl(\CONT_0(E),\CONT_0(E)\bigr)\) with the number
  \(\mult(\Selfmap,\dot\sigma)\) defined in Notation
  \textup{\ref{note:mult_gamma}}.
\end{lemma}

The proof of Lemma~\ref{lem:combinatorial_mult} will finish the
proof of Theorem~\ref{the:Lef_combinatorial} and will occupy
the remainder of this section.

Fix \(\sigma\in SX\) and let \(d\defeq \dim\sigma\) and
\(f\defeq \col(\sigma) \in \poset\).  Let
\[
D_\sigma \defeq
\{x\in \abs{\sigma} \mid \abs{\col}(x)\in CR_f\}.
\]
This is a closed neighbourhood of the barycentre of~\(\sigma\)
that does not intersect the boundary of~\(\sigma\).  Let
\[
D'_\sigma \defeq \{y\in\abs{\sigma}\mid \selfmap(y)\in D_\sigma\}.
\]

Recall that~\(\selfmap\) is a simplicial map \(\Sd X\to X\) for
some subdivision of~\(X\).  Let \(\Sd \sigma\) be the set of
\(d\)\nb-dimensional simplices in the subdivision
of~\(\sigma\).  If \(\tau\in \Sd\sigma\), then either~\(\selfmap\)
maps~\(\tau\) bijectively onto~\(\sigma\) or
\(\abs{\selfmap(\tau)}\cap \abs{\sigma} \subseteq
\partial\abs{\sigma}\).  Hence~\(\abs{\tau}\) and \(D'_\sigma\)
intersect if and only if \(\selfmap(\tau)=\sigma\).  As a result,
\(D'_\sigma\) is a disjoint union of homeomorphic copies of
\(CR_f\), one for each simplex \(\tau\in \Sd\sigma\) with
\(\selfmap(\tau)=\sigma\).  Let \(\tau_1,\dotsc,\tau_k\) be a list
of these simplices.  For \(j=0,\dotsc,k\), define
\(\alpha_j\colon \CONT_0(E) \to \CONT_0(E)\)
by
\[
\alpha_j(\varphi)(t)\defeq
\begin{cases}
  \varphi \Bigl(h^{-1}\bigl(t-r_\lambda \circ
  \abs{\col}\circ \selfmap\circ \bar{q}_\sigma(t)\bigr)\Bigr)
  &\text{if \(\bar{q}_\sigma(t) \in \tau_j \cap D_\sigma'\),}\\
  0&\text{otherwise.}
\end{cases}
\]
Since the supports of the maps~\(\alpha_j\) are disjoint, we
get
\[
[\Xi''_{\selfmap,\dot\sigma}] =
[\alpha_0] + \dotsb + [\alpha_k].
\]
Thus it remains to check that \([\alpha_j] =
(-1)^d\varepsilon_j\), where \(\varepsilon_j=\pm1\) depending
on whether \(\selfmap|_{\abs{\tau_j}}\colon
\abs{\tau_j}\to\abs{\sigma}\) preserves or reverses
orientation.

We claim that the fixed point equation
\[
t= r_\lambda\circ \abs{\col}\circ \selfmap \circ
\bar{q}_\sigma(t)
\]
has a solution \(t= (t_0,\dotsc,t_n)\) in the interior
of~\(R_f\) with \(\bar{q}_\sigma(t)\in \abs{\tau_j} \cap
D_\sigma'\).  In particular, \(\alpha_j\) is not identically
zero.

We may assume right away that \(t=r_\lambda(s)\) for some
\(s\in CR_f\cap \Sigma_f\).  Thus
\begin{alignat*}{2}
  t_i&>0&\qquad&\text{for \(i\in f\), and}\\
  t_i&=-\frac{\lambda-1}{n+1} &\qquad&
  \text{for \(i\in\no\setminus f\).}
\end{alignat*}
Although~\(\bar{q}_\sigma\) is non-linear in general because of
the homogeneous coordinates involved, its restriction to points
of this special form \emph{is} linear: it simply annihilates
the coefficients~\(t_i\) for \(i\notin f\) and rescales the
others by the constant
\[
\frac{(n+1)\lambda}{(d+1)\lambda + (n-d)}
\]
to get a point in~\(E\).  Furthermore, \(\bar{q}_\sigma\) is
invertible between the relevant \(d\)\nb-dimensional
subspaces.  We consider this restriction in the following when
we speak of \(\bar{q}_\sigma^{-1}\).  Thus our fixed point
problem is equivalent to finding \(u\defeq \bar{q}_\sigma(t)
\in \abs{\tau_j} \cap D_\sigma'\) with
\[
u = \bar{q}_\sigma^{-1} \circ
r_\lambda \circ \abs{\col}\circ\selfmap(u).
\]

The map \(A\defeq \bar{q}_\sigma^{-1} \circ r_\lambda \circ
\abs{\col}\circ\selfmap\) is \emph{affine} on
\(\abs{\tau_j}\cap D_\sigma'\) because each of the factors is
affine on the relevant subsets.  Furthermore, \(A\) is
expansive in all directions: the pre-image of~\(\abs{\sigma}\)
is contained in the interior of \(\abs{\tau_j}\subseteq
\abs{\sigma}\).  Hence its inverse map~\(A^{-1}\) is uniformly
contractive.  Thus \(A^{-n}(s)\) converges towards a fixed
point of~\(A\).  This shows that a fixed point exists.

Bott periodicity implies that the inclusion map
\(\CONT_0\bigl(B(\delta)\bigr)\to \CONT_0(\R^n)\) is a
\(\KK\)-equivalence for all \(\delta>0\).  Hence we can compute
the class of~\(\alpha_j\) by restricting it to an arbitrarily
small ball around some point.  Near the fixed point~\(t\)
constructed above, the map
\[
t'\mapsto
h^{-1}\bigl(t'- r_\lambda\circ \abs{\col}\circ \selfmap \circ
\bar{q}_\sigma(t')\bigr)
\]
is smooth.  Hence \([\alpha_j]=\pm1\) in
\(\KK_0\bigl(\CONT_0(E),\CONT_0(E)\bigr)\), where the sign is
the sign of the determinant of the derivative of the above map.
Since~\(h\) preserves orientation, we may omit it without
changing the sign.  In the directions orthogonal to
\(r_\lambda(\Sigma_f)\), the above map acts identically, so
that they contribute no sign either.  Thus the sign is the same
as the sign of the linear map \(\ID-A_0\), where~\(A_0\) is the
linear part of the affine map~\(A\) above.  Since
\(\norm{A_0^{-1}}<1\) and \(\ID-A_0 = -A_0\cdot
(\ID-A_0^{-1})\), this is the same as the sign of \(\det(-A_0)
= (-1)^d \det(A_0)\).  Thus we get \((-1)^d\) if~\(\selfmap\)
preserves orientation and \((-1)^{d+1}\) otherwise.  This
finishes the proof of Lemma~\ref{lem:combinatorial_mult} and
thus of Theorem~\ref{the:Lef_combinatorial}.

\section{Lefschetz invariants for smooth manifolds}
\label{sec:smooth_proof}

In this section, \(X\) is a complete Riemannian manifold.  The
action of~\(G\) on~\(X\) is not quite required to be proper,
but only isometric, or equivalently, the action should
\emph{factor} through a proper action.  (The equivalence of
these two conditions follows because the isometry group
of~\(X\) is a Lie group that acts properly on~\(X\)
provided~\(X\) has only finitely many connected components;
conversely, given a proper action of a Lie group by
diffeomorphisms, there is a complete Riemannian metric for
which the action is isometric.)  Such actions are those for
which Kasparov originally proved duality results in
\cite{Kasparov:Novikov}*{\S4}.  His construction uses Clifford
algebras and differential operators.  First we describe this
dual.  Then we compute the Lefschetz invariant of a suitable
self-map.  This is done once again by plugging all the
ingredients into~\eqref{eq:compute_Lef} and simplifying the
result.  To prepare for the main line of argument, we need a
formula for the result of twisting the Thom isomorphism by a
vector bundle automorphism.  This is the source of the line
bundle \(\sign(\ID_\nu-D_\nu\selfmap)\) in our computation.

\subsection{Description of the Clifford algebra dual}
\label{sec:Cliff_dual}

The complex Clifford algebras of the tangent spaces of~\(X\)
form a locally trivial bundle of finite-dimensional
\(\Z/2\)-graded \Cstar{}algebras \(\Cliff(T^*X)\).  Since~\(G\)
acts isometrically on~\(X\), we get an induced action on
\(\Cliff(T^*X)\) by grading preserving \Star{}algebra
automorphisms.  We let
\[
\dual\defeq \Gamma_0\bigl(X,\Cliff(T^*X)\bigr)
\]
be the \(\Z/2\)\nb-graded \(G\)-\(C^*\)-algebra of sections
of \(\Cliff (T^*X)\).

Let \(\Lambda \defeq \Lambda_\C^*(T^*X)\) be the complexified
exterior algebra bundle of \(T^*X\).  There is a canonical
isomorphism \(\Cliff (T^*X) \cong \Lambda\) that preserves the
grading, inner products, and the \(G\)\nb-action, but not the
algebra structure.  We let \(c \colon \Cliff (T^*X) \to
\operatorname{End}(\Lambda)\) be the resulting representation
by Clifford multiplication.

We describe this on the level of forms.  Let
\(\lambda_\omega\colon \Lambda\to\Lambda\) be the exterior
product with \(\omega\in\Lambda\).  Let~\(i_\omega\) denote the
interior product with~\(\omega\), that is,
\[
i_\omega (u_1\wedge \cdots \wedge u_k) \defeq
\sum_{j=1}^k (-1)^{j-1} \langle\omega,u_j\rangle u_1
\wedge \cdots \wedge \widehat{u_j} \wedge \cdots \wedge u_k.
\]
A simple calculation yields the graded commutator of these
operations:
\[
i_\omega\lambda_\tau + \lambda_\tau i_\omega
= \langle\omega,\tau\rangle.
\]
In particular, \(c(\omega) \defeq \lambda_\omega + i_\omega\)
satisfies \(c(\omega)^2 = \norm{\omega}^2\).  Since \(i_\omega
= \lambda_\omega^*\) as well, \(c\) defines a representation of
the Clifford algebra \(\Cliff(T^*X)\) on~\(\Lambda\).

Let~\(\textup{d}\) be the usual boundary map on differential
forms and let \(\deRham \defeq \textup{d}+\textup{d}^*\).  This
is a \(G\)\nb-equivariant self-adjoint, odd, elliptic
differential operator of order~\(1\), and it commutes with
\(c(\dual)\) up to bounded operators.  Now let \(\Hils\defeq
L^2(\Lambda)\); this is a \(\Z/2\)-graded Hilbert space with a
unitary representation of~\(G\), and~\(c\) yields a
grading-preserving, \(G\)\nb-equivariant \Star{}representation
of \(\dual\) on~\(\Hils\).  The operator~\(\deRham\) is
essentially self-adjoint, and \(M_f (1+\deRham^2)^{-1}\) is
compact for all \(f\in\CONT_0(X)\) because~\(X\) is complete.
Therefore,
\((\Hils,c,\deRham/(1+\deRham^2)^{\nicefrac{1}{2}})\) is a
Kasparov cycle for \(\KK^G_0(\dual,\UNIT)\), which we denote
by~\(D\).

The diagonal embedding \(X\to X\times X\) has a
\(G\)\nb-invariant tubular neighbourhood~\(U\) that is
\(G\)\nb-equivariantly diffeomorphic to the normal bundle
\(TX\) for the embedding.  We can choose the diffeomorphism
\(\Phi\colon TX\to U\) of the form
\[
(x,\xi)\mapsto
\bigl(x,\exp_x(\alpha_x(\norm{\xi})\cdot \xi)\bigr)
\]
for a function \(\alpha\colon X\times \R_{\ge0}\to \R\) that
takes care of a possibly finite injectivity radius.  It is
important that \(\pi_1\circ \Phi\) is the usual projection
\(TX\to X\).

Let \(J_U\subseteq \CONT_0(X)\otimes \dual\) be the
\(G\)\nb-invariant ideal of functions that vanish
outside~\(U\).  We view this as a \(G\)\nb-equivariant
Hilbert module over \(\CONT_0(X)\otimes\dual\) in the usual
way.  We let \(\CONT_0(X)\) act on~\(J_U\) by pointwise
multiplication: \((f_1\cdot f_2)(x,y) \defeq f_1(x)\cdot
f_2(x,y)\).  Identifying \(U\cong TX\), we get a canonical
section of the bundle underlying~\(J_U\), which associates to
\(\Phi(x,\xi)\in U\) the vector
\(\xi/(1+\abs{\xi}^2)^{\nicefrac{1}{2}} \in T_xX\) viewed as an
element of \(\Cliff(T_xX)\).  This defines a
\(G\)\nb-invariant, odd, self-adjoint multiplier~\(F\)
of~\(J_U\) with \(1-F^2\in J_U\).  Hence \((J_U,F)\) yields a
class \(\Theta\in\RKK^G_0\bigl(X;\UNIT,\dual\bigr)\).

Results in \cite{Kasparov:Novikov}*{\S4} show that
\((\dual,\Theta,D)\) as defined above is a Kasparov dual
for~\(X\).

\subsection{A twisted Thom isomorphism}
\label{sec:twisted_Thom}

The computations in this section explain how the line bundle
\(\sign(\ID-D_\nu\selfmap)\) appears in our Lefschetz formula.

Let~\(G\) be a locally compact group, let~\(X\) be a locally
compact proper \(G\)\nb-space, and let \(\pi\colon E\to X\)
be a \(G\)\nb-equivariant real vector bundle over~\(X\) with
\(G\)\nb-invariant inner product.  First we generalise the
construction of \(D\) and~\(\Theta\) above by working
fibrewise.  For each \(x\in X\), the fibre~\(E_x\) has
\(\CONT_0\bigl(E_x,\Cliff(E_x)\bigr)\) as a Kasparov dual, via
classes~\(D_x\) and \(\Theta_x\).  These combine to classes
\begin{align*}
  D_E^X&\in \cRKK^G_0\bigl(X;\CONT_0\bigl(E,\Cliff(E)\bigr),
  \CONT_0(X)\bigr),\\
  \Theta_E^X&\in \cRKK^G_0\Bigl(E;\CONT_0(X),
  \CONT_0\bigl(E,\Cliff(E)\bigr)\Bigr);
\end{align*}
here we tacitly pull back~\(E\) to a bundle on~\(E\) via
\(\pi\colon E\to X\) to form
\[
\dual_E^X \defeq \CONT_0\bigl(E,\Cliff(E)\bigr);
\]
this is a \(G\)\nb-\Cstar{}algebra over~\(E\).

More explicitly, \(D_E^X\) is the de Rham operator along the
fibres of~\(\pi\).  Since~\(E\) is already a vector bundle
over~\(X\), all of \(E\times_X E\) is a tubular neighbourhood
of the diagonal \(E\subseteq E\times_X E\).  Hence we can
simplify~\(\Theta_E^X\); the underlying Hilbert module is
simply
\[
\CONT_0\bigl(E\times_X E,\pi_2^*\Cliff(E)\bigr),
\]
where \(\pi_2\colon E\times_X E\to E\) is the second coordinate
projection.  The group~\(G\) acts in an obvious way, and
\(\CONT_0(E)\) acts by pointwise multiplication via
\(\pi_1\colon E\times_X E\to E\).  We define an essentially
unitary multiplier of \(\CONT_0\bigl(E\times_X
E,\pi_2^*\Cliff(E)\bigr)\), that is, a bounded continuous
section of \(\pi_2^*\Cliff(E)\), by
\[
F(\xi,\eta) \defeq
(1+\norm{\xi-\eta}^2)^{-\nicefrac{1}{2}} \cdot (\xi-\eta).
\]
This determines an operator~\(F\) on \(\CONT_0\bigl(E\times_X
E,\pi_2^*\Cliff(E)\bigr)\), and we get our cycle
\(\Theta^X_E\).  The triple \((\dual_E^X,D_E^X, \Theta_E^X)\)
is a \(G\)\nb-equivariant Kasparov dual for the space~\(E\)
over~\(X\) (see~\cite{Emerson-Meyer:Dualities} for this
relative notion of duality).  This means that we have canonical
isomorphisms
\[
\cRKK^G_*\bigl(E;\pi^*(A), \pi^*(B)\bigr)
\cong
\cRKK^G_*\bigl(X;\CONT_0\bigl(E,\Cliff(E)\bigr)\otimes_X A, B\bigr)
\]
for all \(G\ltimes X\)-\Cstar{}algebras \(A\) and~\(B\).
Actually, since the projection \(E\to X\) is a
\(G\)\nb-equivariant homotopy equivalence,
\[
\cRKK^G_*\bigl(E;\pi^*(A), \pi^*(B)\bigr) \cong
\cRKK^G_*(X;A, B).
\]
The Kasparov duality simply means that \(D_E^X\in
\cRKK^G_*\bigl(X;\dual_E^X,\CONT_0(X)\bigr)\) is invertible and
that the inverse is the element
\[
\tilde\Theta_E^X\in\cRKK^G_*(X;\CONT_0(X),\dual_E^X)
\]
that corresponds to~\(\Theta_E^X\).  Explicitly, the underlying
Hilbert module of~\(\tilde\Theta_E^X\) is~\(\dual_E^X\) with
the usual action of~\(G\) and the representation of
\(\CONT_0(X)\) by pointwise multiplication operators.  The
essentially unitary operator in the Kasparov cycle is the
multiplier~\(F_E^X\) of~\(\dual_E^X\) defined by
\[
E\ni (x,\xi)\mapsto
(1+\norm{\xi}^2)^{-\nicefrac{1}{2}} \xi \in E_x
\subseteq\Cliff(E_x).
\]
The proof that \(D_E^X\) and \(\tilde\Theta_E^X\) are inverse
to each other can be reduced to the case where~\(X\) is a point
using the same trick as in
\cite{LeGall:KK_groupoid}*{Th\'eor\`eme 7.4}.  (The Clifford
algebras allow a kind of Thom isomorphism even if~\(E\) is not
\(\K\)\nb-oriented.)

The following proposition is the entry point for the line
bundle \(\sign(f)\) in our Lefschetz computation.  It is a
refinement of the results in
\cite{Echterhoff-Emerson-Kim:Fixed}*{\S2}.

\begin{proposition}
  \label{pro:index_self-map}
  Let \(f\colon E\to E\) be a \(G\)\nb-equivariant
  isomorphism of vector bundles.  Define
  \[
  f^!\colon
  \CONT_0\bigl(E,\Cliff(E)\bigr) \to
  \CONT_0\bigl(E,\Cliff(E)\bigr),
  \qquad f^!(\varphi)(x,\xi) \defeq
  \varphi\bigl(x,f(\xi)\bigr)
  \]
  for \((x,\xi)\in E\), \(\varphi\in
  \CONT_0\bigl(E,\Cliff(E)\bigr) = \dual_E^X\).  This yields
  \([f^!]\in \cRKK^G_0(X;\dual_E^X,\dual_E^X)\).  The
  composition
  \[
  \CONT_0(X) \xrightarrow{\tilde\Theta_E^X}
  \dual_E^X \xrightarrow{[f^!]}
  \dual_E^X \xrightarrow{D_E^X}
  \CONT_0(X)
  \]
  is the class in \(\KK^G_0\bigl(\CONT_0(X),\CONT_0(X)\bigr)\)
  of the \(G\)\nb-equivariant \(\Z/2\)-graded line bundle
  \(\sign(f)\otimes_\R\C\) over~\(X\).
\end{proposition}

\begin{proof}
  Since~\(f\) is homotopic to the isometry in its polar
  decomposition, we may assume that~\(f\) itself is isometric.

  The Kasparov product of \(\tilde\Theta_E^X\) and~\(f^!\) is
  easy to compute because the latter is a \Star{}homomorphism.
  We get~\(\dual_E^X\) with the usual action of~\(G\) and of
  \(\CONT_0(X)\) by pointwise multiplication -- because
  \(\pi\circ f=\pi\colon E\to X\).  The multiplier~\(F_E^X\)
  above is changed, however, to
  \[
  (x,\xi)\mapsto
  f(\xi)\cdot (1+\norm{f(\xi)}^2)^{-\nicefrac{1}{2}} =
  f(\xi)\cdot (1+\norm{\xi}^2)^{-\nicefrac{1}{2}}.
  \]

  The construction of \(\tilde{\Theta}_E^{X}\) uses the left
  regular representation \(\Cliff (E) \to
  \operatorname{End}(S)\), where \(S= \Cliff (E)\), of the
  Clifford algebra on itself.  The left regular representation
  can be extended to an \emph{irreducible} representation
  \(c\colon \Cliff(E\oplus E^{-}) \cong \Cliff (E)\hat{\otimes}
  \Cliff (E) \to \operatorname{End}(S)\) of \(\Cliff(E\oplus
  E^-)\), where~\(E^-\) denotes~\(E\) with bilinear form
  negated.  That is, \(S\) has the structure of a spinor bundle
  over \(\Cliff(E\oplus E^{-})\).  To see this, we combine the
  right regular representation of~\(E\), which commutes with
  the left regular one, with the grading~\(\gamma\);
  since~\(E\) is contained in the odd part of \(\Cliff(E)\),
  the two maps
  \[
  \Cliff(E)\rightrightarrows \Cliff(E),
  \qquad
  a\mapsto x\cdot a,
  \quad
  a\mapsto \gamma a\gamma \cdot y,
  \]
  for \(x,y\in E\) anti-commute, and the square of the latter
  maps~\(a\) to \(-a\cdot y^2 = -a\norm{y}^2\). The fact that
  we get an irreducible representation follows by counting
  dimensions.

  We use the spinor bundle \(S = \Cliff(E)\) for \(E\oplus
  E^-\) just constructed to compute the line bundle
  \(\sign(f)\).  The point of \(\sign (f)\) is that we have a
  canonical equivariant isomorphism of complex vector bundles
  \begin{equation}
    \label{canonicaliso}
    \sign (A) \otimes_\R (S,c) \cong (S,c'),
    \qquad \varphi\otimes a \mapsto \varphi(a),
  \end{equation}
  where~\(c'\) as usual denotes the Clifford multiplication
  twisted by \(f\oplus \ID\).  This isomorphism is equivariant
  with respect to the ordinary (untwisted) Clifford
  multiplication \((x,y) \mapsto c(x,y) \otimes \ID\) on its
  domain, and the twisted action \((x,y) \mapsto c'(x,y) =
  c(f(x), y)\) on its co-domain.  It is easy to check that the
  isomorphism~\eqref{canonicaliso} respects ordinary right
  Clifford multiplication (because the grading~\(\gamma\) can
  be built out of left Clifford multiplication) whence we have
  a \(G\)\nb-equivariant isomorphism of right
  \(\CONT_0\bigl(X, \Cliff (E)\bigr)\)-modules respecting left
  Clifford multiplication.  This shows that the twisted
  Clifford multiplication that appears in \(\tilde\Theta_E^X
  \otimes_{\dual_E^X} [f^!]\) is isomorphic to the standard
  Clifford multiplication on \(\sign(f)\otimes \Cliff(E)\).
  Thus
  \[
  \tilde\Theta_E^X \otimes_{\dual_E^X} [f^!] =
  [\sign(f)] \otimes_{\CONT_0(X)} \tilde\Theta_E^X.
  \]
  Since \(D_E^X=(\tilde\Theta_E^X)^{-1}\) by the untwisted Thom
  isomorphism, the product with \(D_E^X\) yields the class of
  \(\sign(f)\) in
  \(\cRKK^G_0\bigl(X;\CONT_0(X),\CONT_0(X)\bigr)\).
\end{proof}

\subsection{Computing the Lefschetz invariant}
\label{sec:smooth_Lefschetz}

Now we compute the Lefschetz map for a smooth
\(G\)\nb-equivariant map \(\selfmap\colon X\to X\) that
satisfies the prerequisites of Theorem~\ref{the:Lef_smooth},
using the Kasparov dual involving Clifford algebras described
above.

Throughout this section, we let \(Y\defeq \Fix(\selfmap)\);
this is a closed submanifold of~\(X\) by assumption.
Let~\(\nu\) be its normal bundle; this is a vector bundle
over~\(Y\) via the projection map \(\pi\colon \nu\to Y\).
Since~\(Y\) is a closed submanifold, it has a tubular
neighbourhood~\(V\), that is, \(V\cong\nu\) via a
\(G\)\nb-equivariant diffeomorphism whose restriction
to~\(Y\) is the zero section of~\(\nu\).  Extending functions
by~\(0\) outside~\(V\), we get canonical embeddings such as
\[
j\colon \CONT_0(\nu)\to \CONT_0(X),
\qquad
j\colon \CONT_0\bigl(\nu,\Cliff(TX|_\nu)\bigr)\to
\CONT_0(X,\Cliff(TX)\bigr) = \dual.
\]

Recall that
\[
\Lef(\selfmap)
= \overline{\Theta} \otimes_{\CONT_0(X)\otimes\dual}
[\mu_\selfmap]\otimes_\dual D.
\]
Lemma~\ref{lem:compute_T_on_map} shows that
\(\mu_\selfmap\colon \CONT_0(X,\dual)\cong
\CONT_0(X)\otimes\dual\to\dual\) is given by
\[
\mu_\selfmap(\varphi)(x) = \varphi(\selfmap(x),x)
\in \Cliff(T_xX)
\qquad
\text{for all \(\varphi\in\CONT_0(X,\dual)\), \(x\in X\).}
\]
We are going to compose this \Star{}homomorphism
with~\(\overline{\Theta}\) and simplify the result:

\begin{lemma}
  \label{lem:simplify_Theta_mu_selfmap}
  The composition \(\overline{\Theta}
  \otimes_{\CONT_0(X)\otimes\dual} [\mu_\selfmap] \in
  \KK^G_0(\CONT_0(X),\dual)\) is equal to the composite
  \begin{multline*}
    \CONT_0(X) \xrightarrow{r_Y}
    \CONT_0(Y) \xrightarrow{\sign(\ID_\nu-D_\nu\selfmap)}
    \CONT_0(Y) \xrightarrow{\tilde\Theta_\nu^Y}
    \CONT_0\bigl(\nu,\Cliff(\nu)\bigr) \\\xrightarrow{i}
    \CONT_0\bigl(\nu,\Cliff(TX|_\nu)\bigr) \xrightarrow{j}
    \CONT_0\bigl(X,\Cliff(TX)\bigr) =
    \dual
  \end{multline*}
  Here~\(r_Y\) is the restriction map;
  \(\sign(\ID_\nu-D_\nu\selfmap) \in
  \KK^G_0\bigl(\CONT_0(Y),\CONT_0(Y)\bigr)\) is the class
  associated to the corresponding line bundle;
  \(\tilde\Theta_\nu^Y\) is as in
  \textup{\S\ref{sec:twisted_Thom}}; the map~\(i\) is induced
  by the embedding
  \[
  \Cliff(\nu) \to \Cliff(\nu)\otimes\Cliff(TY) \cong
  \Cliff(\nu\oplus TY) \cong \Cliff(TX|_\nu),
  \]
  where the first map uses the unit element in \(\Cliff(TY)\);
  and~\(j\) is induced by the embedding of~\(\nu\) in~\(X\).
\end{lemma}

\begin{proof}
  We have described~\(\Theta\) and hence~\(\overline{\Theta}\)
  by an explicit cycle.  To get
  \[
  \overline{\Theta}\otimes_{\CONT_0(X,\dual)} [\mu_\selfmap]
  \in \KK^G_0(\CONT_0(X),\dual),
  \]
  we restrict this cycle to the graph of~\(\selfmap\).  Recall
  that~\(\Theta\) is supported in a certain
  \(G\)\nb-invariant open neighbourhood~\(U\) of the diagonal
  of~\(X\) in \(X\times X\).  We let
  \[
  U'\defeq \{x\in X\mid (x,\selfmap x)\in U\};
  \]
  this is a \(G\)\nb-invariant neighbourhood of the fixed
  point submanifold \(Y\defeq \Fix(\selfmap)\).  Restriction to
  the graph of~\(\selfmap\) replaces~\(J_U\) by the ideal
  \[
  J_U' \defeq \bigl\{f\in\CONT_0\bigl(X,\Cliff(TX)\bigr) \bigm|
  \text{\(f(x) = 0\) unless \(\bigl(x,\selfmap(x)\bigr)\in
    U\)}\bigr\}
  \]
  in~\(\dual\).  The group~\(G\) acts on~\(J_U'\) in the
  obvious way, and \(\CONT_0(X)\) acts on~\(J_U'\) by pointwise
  multiplication: \((f\cdot\varphi)(x)\defeq f(x)\cdot
  \varphi(x)\).  The multiplier~\(F\) described in
  \S\ref{sec:Cliff_dual} yields the multiplier
  \[
  F'(x) \defeq F(x,\selfmap x) \in T_xX \subseteq
  \Cliff(T_xX),
  \]
  where \(F(x,\selfmap x)\) is the pre-image of~\(\selfmap(x)\)
  under a suitably rescaled exponential map at~\(x\).

  In the construction of the Kasparov dual for~\(X\), we may
  choose~\(U\) to be an arbitrarily small neighbourhood of the
  diagonal.  For a suitable choice of~\(U\), the
  neighbourhood~\(U'\) of~\(Y\) will be contained in~\(V\), a
  tubular neighbourhood around~\(Y\).  We assume this from now
  on.

  Let \(\pi\colon V\cong \nu\to Y\) be the retraction from the
  Tubular Neighbourhood Theorem.  Since this is a
  \(G\)\nb-equivariant deformation retraction, pointwise
  multiplication by \(f(x)\) and \(f\circ\pi(x)\) is
  \(G\)\nb-equivariantly homotopic.  Therefore, we may
  replace the action of \(\CONT_0(X)\) on~\(J_U'\) by the one
  of pointwise multiplication with \(f\circ\pi\) for
  \(f\in\CONT_0(X)\).  This factors
  \(\overline{\Theta}\otimes_{\CONT_0(X,\dual)}
  [\mu_\selfmap]\) through the restriction map
  \(\CONT_0(X)\to\CONT_0(Y)\).

  Equip~\(\nu\) with some Euclidean inner product and transport
  the resulting norm to~\(V\) via the diffeomorphism
  \(V\cong\nu\).  Since~\(U\) is a neighbourhood of the
  diagonal, \(U'\) is a neighbourhood of~\(Y\) in~\(X\).  Let
  \(\varrho\colon Y\to\R_{>0}\) be a \(G\)\nb-invariant
  function with \(x\in U'\) for all \(x\in V\) with
  \(\norm{x}\le\varrho\circ\pi(x)\).  Since
  \(F\bigl(x,\selfmap(x)\bigr)\) does not vanish unless \(x\in
  Y\), we can rescale~\(F'\) such that \(F'(x)^2=1\) for all
  \(x\in U'\) with \(\norm{x}\ge\varrho\circ\pi(y)\).  This
  yields a homotopic cycle.  Now we may restrict~\(J_U'\) to
  the smaller ideal of elements of~\(\dual\) supported in
  \[
  U'' \defeq \{x\in V\mid \norm{x}< \varrho\circ\pi(x)\}
  \]
  because the operator~\(F'\) has become unitary on the
  complement, resulting in our Kasparov cycle being degenerate
  there.  This neighbourhood is another tubular neighbourhood
  of~\(Y\) by a rescaling.  Changing our tubular neighbourhood,
  we can therefore achieve that~\(\varrho\) becomes the
  constant function~\(1\) to simplify.

  We define another function
  \[
  \tilde{F}\colon V\cong\nu\to TX,
  \qquad
  \tilde{F}(y,\xi) \defeq (\ID-D\selfmap)(\xi)
  \]
  for \(y\in Y\), \(\xi\in\nu_y\); here we use some
  \(G\)\nb-equivariant section for the quotient map
  \(TX\prto\nu\).  We join~\(\tilde{F}\) and~\(F'\) by the
  linear homotopy \(tF'+(1-t)\tilde{F}\).  Since
  \(\selfmap(x)\approx x+D\selfmap(x)\) for~\(x\) near~\(Y\),
  there is a neighbourhood~\(\tilde{V}\) of~\(Y\) such that
  \(tF'+(1-t)\tilde{F}\) is invertible on \(\tilde{V}\setminus
  Y\).  Rescaling~\(\tilde{F}\) and~\(F'\) first, so that they
  become unitary outside~\(\tilde{V}\), and also rescaling the
  above homotopy, we connect our Kasparov cycle to
  \((\dual|_{\tilde{V}},\tilde{F})\), with \(\CONT_0(X)\)
  acting by pointwise multiplication combined with~\(\pi\):
  \((f_1\cdot f_2)(x) = f_1\circ\pi(x)\cdot f_2(x)\).

  Our computation so far shows that the cycle that defines
  \(\overline{\Theta} \otimes_{\CONT_0(X)\otimes\dual}
  [\mu_\selfmap]\) is homotopic to \(r_Y\otimes_{\CONT_0(Y)}
  \tilde\Theta_\nu^Y \otimes_{\CONT_0(\nu,\Cliff(\nu))}
  [(\ID_\nu-D_\nu\selfmap)^!]
  \otimes_{\CONT_0(\nu,\Cliff(\nu))} (j\circ i)\) -- the
  changes in our choice of the tubular neighbourhood above do
  not matter.  Finally, Proposition~\ref{pro:index_self-map}
  yields
  \[
  \tilde\Theta_\nu^Y \otimes_{\CONT_0(\nu,\Cliff(\nu))}
  [(\ID_\nu-D_\nu\selfmap)^!] =
  \sign(\ID_\nu-D_\nu\selfmap) \otimes_{\CONT_0(Y)}
  \tilde\Theta_\nu^Y
  \]
  because \(\tilde\Theta_\nu^Y\) and~\(D_\nu^Y\) are inverse to
  each other.
\end{proof}

To get \(\Lef(\selfmap)\), we must compose the product computed
in Lemma~\ref{lem:simplify_Theta_mu_selfmap} with
\(D\in\KK^G_0(\dual,\UNIT)\).  To begin with, we compose~\(D\)
with the class of the \Star{}homomorphism \(j\circ i\colon
\CONT_0\bigl(\nu,\Cliff(\nu)\bigr) \to \dual\).  This yields
the class of the operator~\(\deRham\) on the space of sections
\(L^2\bigl( \Lambda^*_\C (T^*X)\bigr)\) over~\(X\) of the
bundle \(\Lambda^*_\C (T^*X)\), with
\(\CONT_0\bigl(\nu,\Cliff(\nu)\bigr)\) acting by left
multiplication; here we extend such functions by~\(0\) outside
\(\nu\cong V\) and use the embedding \(\Cliff(\nu)\to
\Cliff(TX|_\nu)\).  This Kasparov cycle is highly degenerate:
we may restrict to the subspace of differential forms in
\(L^2\bigl( \Lambda^*_\C (T^*X)\bigr)\) that vanish
outside~\(\nu\); the result is isomorphic to the bundle of
forms on~\(\nu\) with respect to a \emph{complete} Riemannian
metric on~\(\nu\).  The restriction of~\(\deRham\) to
differential forms on~\(X\) which are supported in~\(\nu\) is
homotopic to the Euler operator for~\(\nu\) because both
operators are pseudodifferential and have the same principal
symbol, up to the isomorphism involved in changing the metric.
Thus we now want to compose the Euler operator on~\(\nu\) with
\(\tilde{\Theta}_Y^\nu\).

We may split the class of the Euler operator on~\(\nu\) as a
Kasparov product of the class \(D_\nu^Y\in
\KK^G_0\bigl(\CONT_0\bigl(\nu,\Cliff(\nu)\bigr),\CONT_0(Y)\bigr)\)
and the class in \(\KK^G_0(\CONT_0(Y),\UNIT)\) of the Euler
operator for~\(Y\):
\[
(ji)^*(D) = D_\nu^Y \otimes_{\CONT_0(Y)} \Eul_Y.
\]
This is proved like the corresponding assertion for Dirac
operators and amounts again to the corresponding fact about the
symbols.  Finally, the class \(D_\nu^Y\) cancels
\(\tilde\Theta_\nu^Y\) in
Lemma~\ref{lem:simplify_Theta_mu_selfmap} and yields
\(\Lef(\selfmap) = r_Y \otimes_{\CONT_0(Y)} [\sign(\ID_\nu -
D_\nu\selfmap)] \otimes_{\CONT_0(Y)} \Eul_Y\) as asserted.
This finishes the proof of Theorem~\ref{the:Lef_smooth}.

\section{Conclusion and outlook}
\label{sec:conclusion}

We have used duality in bivariant \(\KK\)\nb-theory to refine
the Lefschetz number of a self-map to an equivariant
\(\K\)\nb-homology class, and we have computed this invariant
for suitable self-maps of simplicial complexes and smooth
manifolds.  In both cases, the Lefschetz invariant only sees a
small neighbourhood of the fixed point subset.

In the simplicial case, the equivariant Lefschetz invariant is
a \(0\)\nb-dimensional object in the sense that it is a
difference of two equivariant \Star{}homomorphisms to a
\(C^*\)\nb-algebra of compact operators.  This is a special
feature of Lefschetz invariants of self-maps.  Therefore, it is
interesting to extend the computation of the Lefschetz
invariant to more general classes in
\(\KK^G_*\bigl(\CONT_0(X),\CONT_0(X)\bigr)\) or even
\(\RKK^G_*(X;\CONT_0(X),\C)\).  Since the Lefschetz map is a
split surjection on the latter group, this will necessarily
lead to more complicated \(\K\)\nb-homology classes.

A geometric computation of the Lefschetz map in this case
requires descriptions of the relevant Kasparov groups in terms
of geometric cycles, and the use of a dual that is appropriate
to this situation.  This geometric computation of the Lefschetz
invariant will be the subject of a forthcoming article.

\end{document}